\documentclass[11pt]{article}
\usepackage{amssymb}
\usepackage{amsmath}

\newenvironment{proof}{\noindent {\bf Proof }}
{\hfill $\bullet$ \vspace{0.25cm}}

\def\P{{\mathbb P}}
\def\R{{\mathbb R}}
\def\Z{{\mathbb Z}}

\def\N{{\mathbb N}}

\def\F {{\mathcal F}}

\def\s {{\sigma}}

\def\SS{ {\cal S}^\xi}

\newtheorem{theo}{Theorem}
\newtheorem{prop}{\indent Proposition}

\newtheorem{rem}{\indent Remark}
\newtheorem{lem}{\indent Lemma}

\newtheorem{cor}{\indent Corollary}

\newtheorem{ex}{\indent Example}

\title{Infinite systems of interacting chains with memory of variable length - a stochastic model for biological neural nets}

\date{February 19, 2013}
\author{A.~Galves \and E.~L\"ocherbach}

\begin{document}

\maketitle 

\begin{center}
{\it Dedicated to Errico Presutti, frateddu e mastru}
\end{center}

\begin{abstract}
We consider a new class of non Markovian processes with a countable number of interacting components. At each time unit, each component can take two values, indicating if it has a spike or not at this precise moment. The system evolves as follows. For each component, the probability of having a spike at the next time unit depends on the entire time evolution of the system after the last spike time of the component. This class of  systems extends in a non trivial way both the interacting particle systems, which are Markovian, \nocite{Spitzer1970}  and the stochastic chains with memory of variable length which have finite state space.\nocite{rissanen}
These features make it suitable to describe the time evolution of biological neural systems. We construct a stationary version of the process by using a probabilistic tool which is a Kalikow-type decomposition either in random environment or in space-time. This construction implies uniqueness of the stationary process. 
Finally we consider the case where the interactions between components are given by a critical directed Erd\"os-R\'enyi-type random graph with a large but finite number of components. In this framework we obtain an explicit upper-bound for the correlation between successive inter-spike intervals which is compatible with previous  empirical findings.
\end{abstract}



\section{Introduction}
A biological neural system has the following characteristics. It is a system with a huge (about $10^{11}$) number of interacting components, the neurons. 
This system evolves in time, and its time evolution is not described by a Markov process (Cessac 2011\nocite{cessac_discrete_time}). In particular, the times between successive spikes of a single neuron are not exponentially distributed (see, for instance, Brillinger 1988\nocite{Brillinger1988}). 

This is the motivation for the introduction of the class of models that we consider in the present paper. To cope with the problem of the large number of components it seems natural to consider infinite systems with a countable number of components. In this new class of stochastic systems, each component depends on a variable length portion of the history. Namely, the spiking probability of a given neuron depends on the accumulated activity of the system after its last spike time.  
This implies that the system is not Markovian. The time evolution of each single neuron looks like a stochastic chain with memory of variable length, even if the influence from the past is actually of infinite order. This class of systems represents a non trivial extension of the class of interacting particle systems introduced in 1970 by Spitzer. It is also a non trivial extension of the class of stochastic chains with memory of variable length introduced in 1983 by Rissanen\nocite{rissanen}. 

The particular type of dependence from the past considered here is motivated both by empirical as well as theoretical considerations. 

From a theoretical point of view, Cessac (2011)\nocite{cessac_discrete_time} suggested the same kind of dependence from the past. In the framework of leaky integrate and fire models, he considers a system with a finite number of membrane potential processes. The
image of this process in which only the spike times are recorded is a stochastic chain of infinite order where each neuron has to look back into the past until its last spike time. Cessac's process is a finite dimensional version of the model considered here. 

Finite systems of point processes in discrete or continuous time aiming to describe biological neural systems have a long history whose starting points are probably Hawkes (1971)\nocite{hawkes71} from a probabilistic point of view and  Brillinger (1988)\nocite{Brillinger1988} from a statistical point of view, see also the interesting paper by Krumin et al. (2010)\nocite{krumin} for a review of the statistical aspects. For non-linear Hawkes processes, but in the frame of a finite number of components, Br\'emaud and Massouli\'e (1994)\nocite{BremaudMassoulie96} address the problem of existence, uniqueness and stability. M{\o}ller and coauthors propose a perfect simulation algorithm in the linear case, see M{\o}ller and Rasmussen (2005)\nocite{MollerRasmussen}. 
In spite of the great interest in Hawkes processes during the last years, especially in association with modeling problems in finance and biology, all the studies are reduced to the case of systems with a finite number of components. Here we propose a new approach which enables us to deal also with infinite systems with a countable number of components, without any assumption of the type linearity or attractiveness.  

This paper is organized as follows. In Section \ref{section:def} we state two Theorems proving the existence and uniqueness of infinite systems of interacting chains with memory of variable length, under suitable conditions. 
Our main technical tool is a Kalikow-type decomposition of the infinite order transition probabilities which is a non trivial extension of previous results of the authors in the case of Markovian systems, cf. Galves et al. (2013)\nocite{gglo2013}. 
The decomposition considered here has two major differences with respect to what has been done before. Firstly this is due to the non-Markovian nature of the system. Secondly, and most importantly, the structure of the transition laws leads to the need of either considering a decomposition depending on a random environment or considering a space-time decomposition. Using the Kalikow-type decomposition we prove the existence, the uniqueness as well as a property of loss of memory of the stationary process. 

In Section \ref{sec:stat} we study the correlation between successive inter-spike intervals (ISI). This aims at explaining empirical results presented in the neuroscientific literature. Gerstner and Kistler (2002)\nocite{Gerstner}, quoting Goldberg et al. (1964)\nocite{Goldberg}, observe that in many experimental setups the empirical correlation between successive inter-spike intervals is very small ``indicating that a description of spiking as a stationary renewal process is a good approximation".  However, Nawrot et al. (2007) find statistical evidence that neighboring inter-spike intervals are correlated, having negative correlation. We show that we can account for these apparently contradictory facts within our model. This requires the introduction of a new setup in which the synaptic weights define a critical directed Erd\"os-R\'enyi random graph with a large but finite number of components. We obtain in Theorem \ref{theo:4} an explicit upper bound for the correlations involving the number of components of the system, as well as the typical length of one inter-spike interval. For a system having a large number of components, our result is compatible with the discussion in Gerstner and Kistler (2002). Gerstner and Kistler (2002) deduce from this that spiking can be described by a renewal process. However, for systems with a small number of components, the correlation might as well be quite big, as reported by Nawrot et al. (2007) who show that neighboring inter-spike intervals are negatively correlated. Therefore, both features are captured by our model, depending on the scale we are working in. 

The proofs of all the results are presented in Sections \ref{sec:3}, \ref{sec:4} and \ref{sec:proofstat}.

\section{Systems of interacting chains with memory of variable length: Existence, uniqueness and loss of memory}\label{section:def}
We consider a stochastic chain $(X_t  )_{ t \in \Z }$ taking values in $ \{ 0, 1 \}^I $ for some countable set of neurons $I ,$ defined on a suitable probability space $ ( \Omega , { \cal A} , P ) .$ For each neuron $i$ at each time $t \in \Z,$  $ X_t (i) = 1 $ if neuron $i$ has a spike at that time $t$, and $X_t(i) = 0 $ otherwise. The global configuration of neurons at time $t$ is denoted $X_t = (X_t (i) , i \in I ) .$ We define the filtration
$$ \F_t = \s ( X_s ,  s \in \Z, s \le t   ) , t \in \Z .$$ 
For each neuron $i \in I$ and each time $t \in \Z $ let
\begin{equation}
L_t^i = \sup \{ s < t  : X_s (i) = 1  \} 
\end{equation} 
be the last spike time of neuron $i$ strictly before time $t.$ We introduce a family of ``synaptic" weights $W_{j \to i} \in  \R , $ for $j \neq i ,$ $W_{j \to j } = 0 $ for all $j.$  $W_{j\to i}$ is the ``synaptic weight of neuron $j$ on neuron $i$". We suppose that the synaptic weights have the following property of uniform summability. \begin{equation}\label{eq:summable}
\sup_{ i \in I} \sum_j |W_{j \to i }| < \infty .
\end{equation}

Now we are ready to introduce the dynamics of our process. At each time $t ,$ conditionally on the whole past, sites update independently. This means that for any finite subset $J \subset I ,$ $a_i \in \{ 0, 1 \} , i \in J, $ we have
\begin{equation}\label{eq:dyn}
P ( X_t (i ) = a_i , i \in J | \F_{t-1} ) = \prod_{i \in J} P ( X_t (i ) = a_i | \F_{t-1}).
\end{equation}
Moreover, the probability of having a spike in neuron $i$ at time $t$ is given by 
\begin{equation}\label{eq:transition2}
 P( X_t (i ) = 1 | { \cal F}_{t- 1}  ) = \phi_i  \left( \sum_j  W_{ j \to i} \sum_{ s = L_t^i}^{ t- 1} g_j (t-s) X_s (j) , t - L_t^i    \right) ,
\end{equation}
where $\phi_i : \R \times \N   \to [ 0, 1 ]$ and $ g_j : \N \to \R_+ $ are measurable functions for all $ i \in I , j \in I .$ We assume that $\phi_i$ is uniformly Lipschitz continuous, i.e. there exists a positive constant $ \gamma $ 
such that for all $ s, s' \in \R ,$ $n \in \N,$ $i \in I, $
\begin{equation}\label{eq:Lip}
| \phi_i (s, n ) - \phi_i  ( s', n  ) | \le \gamma  |s - s' |   .
\end{equation}    

Observe that in the case where the function $\phi_i $ is increasing with respect to the first coordinate, the contribution of components $j$ is either excitatory or inhibitory, depending on the sign of $W_{j \to i} .$ This is reminiscent of the situation in biological neural nets in which neurons can either stimulate or inhibit the expression of other neurons.

It is natural to ask if there exists at least one stationary chain which is consistent with the above dynamics, and if so, if this process is unique. In what follows we shall construct a probability measure $P $ on the configuration space $\Omega = \{ 0 , 1 \}^{ I \times \Z } $ of all space-time configurations of spike trains, equipped with its natural sigma algebra ${\cal A} .$ On this probability space, we consider the canonical chain $ (X_t)_{ t \in \Z } $ where for each neuron $i$ and each time $t,$ $X_t (i) ( \omega ) = \omega_t (i)  $ is the projection of $ \omega $ onto the $(i, t ) $ coordinate of $ \omega . $ 

For each neuron $i,$ we introduce
$$ {\cal V}_{\cdot \to i} = \{ j \in I, j \neq i  : W_{j \to i } \neq 0 \} ,$$
the set of all neurons that have a direct influence on neuron $i.$ Notice that in our model,  ${\cal V}_{\cdot \to i}  $ can be both finite or infinite. We fix a growing sequence $( V_i (k ) )_{ k \geq - 1 } $ of subsets of $ I $ such that $V_i ( - 1 ) = \emptyset ,$ $ V_i ( 0 ) =  \{ i \}  ,$ $ V_i (k ) \subset  V_i ( k+1) ,$ $  V_i (k ) \neq V_i ( k+1)$ if $ V_i (k ) \neq {\cal V}_{\cdot \to i}   \cup \{ i \}$ and $\bigcup_k V_i (k) ={\cal V}_{\cdot \to i}  \cup \{ i \} .$

We consider two types of systems. The first system incorporates spontaneous spike times, see Condition (\ref{eq:delta}) below. These spontaneous spikes can be interpreted as external stimulus or, alternatively, as autonomous activity of the brain. The existence and uniqueness of this class is granted in our first theorem. 

\begin{theo}\label{theo:3}[Existence and uniqueness in systems with spontaneous spikes]\\
Grant conditions (\ref{eq:summable}) and (\ref{eq:Lip}). Assume that the functions $\phi_i  $ and $g_j$ satisfy moreover the following assumptions:
\begin{enumerate}
\item
There exists $ \delta > 0 $
such that  for all $ i \in I, s \in \R  , n \in \N ,$ 
\begin{equation}\label{eq:delta}
\phi_i  ( s,n ) \geq \delta .
\end{equation}
\item 
We have that
\begin{equation}\label{eq:G}
  G (1) + \sum_{ n = 2 }^\infty ( 1 - \delta)^{ n - 2} n^2  G ( n) < \infty ,
\end{equation}
where $G  ( n) = \sup_i \sum_{ m = 1}^n  g_i (m) $ and where $ \delta $ is as in condition 1.
\item
We have fast decay of the synaptic weights, i.e.
\begin{equation}\label{eq:veryfast}
\sup_i \, \sum_{ k \geq 1 } | V_i (k ) | \left( \sum_{ j \notin V_i (k - 1 ) } |W_{j \to i }| \right) < \infty . 
\end{equation}
\end{enumerate} 
Then under these Conditions (\ref{eq:delta})--(\ref{eq:veryfast}), there exists a critical parameter $\delta_* \in ] 0, 1 [ $ such that for any $ \delta > 
 \delta_* , $ there exists a unique probability measure $P  $ on $\{ 0 , 1 \}^{ I \times \Z } ,$ under which the canonical chain satisfies (\ref{eq:dyn}) and (\ref{eq:transition2}). 
\end{theo}

\begin{rem}\label{rem:1}
The stochastic chain $ (X_t)_{t \in \Z}$ introduced in Theorem \ref{theo:3} is a chain having memory of infinite order (cf. Doeblin and Fortet 1937\nocite{doefor37}, Harris 1955\nocite{har55}, Berbee 1987, Bressaud, Fern\'andez and Galves 1999, Johannson and \"Oberg 2003 and Fern\'andez and Maillard 2004\nocite{fern_maillard}). The setup we consider here extends what has been done in the above cited papers. First of all, the chain we consider takes values in the infinite state space $ \{ 0, 1 \}^I .$ Moreover, in Theorem \ref{theo:3} no summability assumption is imposed on the functions $g_j .$ In particular, the choice $ g_j (t) \equiv 1 $ is possible. This implies that the specification of the chain is not continuous. More precisely, introducing 
$$ p_{(i,t)} ( 1 | x ) = \phi_i \left( \sum_{ j } W_{ j \to i } \sum_{ s = L_t^i ( x) }^{t- 1 } g_j ( t-s ) x_j ( j) , t - L_t^i \right), $$ 
where $ L_t^i (x) = \sup \{ s < t : x_s (i ) = 1 \} ,$
we have that
$$ \sup_{ x , y \; : x= y \; \mbox{\small on } V_i ( k) \times [ t - k , t -  1 ] } | p_{(i,t)} ( 1 | x ) - p_{(i,t)} ( 1 | y  ) | \not\to 0 \mbox{ as } k \to \infty  $$
in the case $ g_j (t) \equiv 1 $ for all $j,$ which can be seen by taking configurations $x$ and $y$ such that $ L_t^i (x) < t - k  $ and $ L_t^i (y ) < t-k .$ 
A similar type of discontinuity has been considered in Gallo (2011)\nocite{Gallo2011} for stochastic chains with memory of variable length taking values in a finite alphabet. 
\end{rem}

As an illustration of Theorem \ref{theo:3} we give the following example of a system with interactions of infinite range. 

\begin{ex}
We give an example of a system satisfying the assumptions of Theorem \ref{theo:3}. Take $ I = \Z^d ,$ $g_j ( s) = 1$ for all $j, s,$ and $ W _{ i \to j } = \frac{ 1 }{ \|j- i\|_1^{2 d + \alpha} }$ for some fixed $\alpha > 1,$ where $\| \cdot \|_1$ is the $L^1-$norm of $\Z^d .$  In this case, if we choose $ V_i (k ) = \{ j \in \Z^d = \| j - i \|_1 \le k \}, $ we have $ | V_i (k ) | = (k + 1)^d , $ and condition (\ref{eq:veryfast}) is satisfied, since
\begin{multline*}
\sum_{ k \geq 1 } | V_i (k ) | \left( \sum_{ j \notin V_i (k - 1 ) } |W_{j \to i }| \right) =  \sum_{ k \geq 1 }  ( k + 1 )^d \sum_{ l = k }^\infty  card \{ j : \| j - i \|_1 = l \} \, \frac{1}{l^{2 d +\alpha}} \\
\le C(d) \sum_{ k \geq 1 }  ( k + 1 )^d \sum_{ l = k }^\infty \frac{ l^{d-1}}{l^{2 d +\alpha}}
\le \frac{C(d)}{d+ \alpha} \sum_{ k \geq 1 } \frac{( k + 1 )^d }{(k-1)^{ d + \alpha }} < \infty ,
\end{multline*}
as $ \alpha  > 1 .$ 
\end{ex}

The next theorem deals with the second type of system. Now we don't assume a minimal spiking rate. But additionally to the fast decay of the synaptic weights we also assume a sufficiently fast decay of the aging factor $g_j , $ see Condition (\ref{eq:veryveryfast}) below. This additional assumption implies that the specification of the chain is continuous. This is the main difference with the setup of Theorem \ref{theo:3}.

\begin{theo}\label{theo:2}[Existence and uniqueness in systems with uniformly summable memory]
Suppose that $\phi_i ( s,n)  = \phi_i (s)  $ does not depend on $n.$ Assume conditions (\ref{eq:summable}) and (\ref{eq:Lip}) and suppose moreover that 
\begin{equation}\label{eq:veryveryfast}
\sup_i \, \sum_{ k \geq 0 }( k + 1) \cdot | V_i (k ) | \left( \sum_{ j \notin V_i (k - 1 ) } |W_{j \to i }| \sum_{n =  1  }^\infty g_j (n) +\sum_{ j \in V_i ( k- 1 ) } |W_{j \to i }| \sum_{n = k \vee 1  }^\infty g_j (n) \right) < \frac{1}{ \gamma }  ,
\end{equation}
where $\gamma $ is given in (\ref{eq:Lip}).

Then there exists a unique probability measure $P $ on $\{ 0 , 1 \}^{ I \times \Z } $ such that under $P ,$ the canonical chain satisfies (\ref{eq:dyn}) and (\ref{eq:transition2}). 
\end{theo}

Now, for any $ s < t \in \Z, $ let $ X_s^t (i) = ( X_s (i), X_{ s+1} ( i ) , \ldots , X_t (i) ) $ the trajectory of $X(i)$ between times $s$ and $t.$ As a byproduct of the proof of Theorems \ref{theo:3} and \ref{theo:2} we obtain the following loss of memory property.

\begin{cor}\label{cor:mixing}
\begin{enumerate}
\item
Under the assumptions of either Theorem \ref{theo:3} or Theorem \ref{theo:2}, there exists a non increasing function $\varphi : \N \to \R_+ ,$
such that for any $ 0 < s < t \in \N $ the following holds. For all $i \in I, $ for all bounded measurable functions $f : \{ 0, 1 \}^{   [ s, t ] } \to \R_+ ,$
\begin{equation}\label{eq:lossofmemory}
 \big| E [ f ( X_s^t (i) )  | {\cal F}_0  ] - E [ f ( X_s^t (i))] \big| \le   \, ( t-s +1 ) \, \| f \|_\infty \, \varphi (s)    .
\end{equation}
Moreover, $\varphi ( n)  \le C \frac{1}{n- 1} $ for some fixed constant $C. $ 
\item
Under the assumptions of Theorem \ref{theo:2}, suppose moreover that there exists a constant $C > 0 $ such that 
\begin{equation}\label{eq:exponentialg}
g_j (n ) \le C e^{ - \beta n } \mbox{ and }  \sup_{i } \sum_{ j \notin V_i (n) } | W_{ i \rightarrow j } |   \le C e^{ - \beta n },
\end{equation}
for all $j \in I , n \in \N ,$ for some $\beta > 0 .$   

Then there exists a critical parameter $\beta_* $ such that if $ \beta > \beta_* , $  (\ref{eq:lossofmemory}) holds with 
\begin{equation}\label{eq:lossexp}
 \varphi ( s) = C \varrho^s \mbox{ for some } \varrho \in ] 0 , 1 [  \mbox{ depending only on $\beta .$}
\end{equation}
\end{enumerate}  
\end{cor}

The proof of Theorem \ref{theo:3} is given in Section \ref{sec:3} below. It is based on a conditional Kalikow-type decomposition of the transition probabilities $\phi_i , $ where we decompose with respect to all possible interaction neighborhoods of site $i .$ A main ingredient is the construction of an associated branching process in random environment. The setup of Theorem \ref{theo:2} is conceptually less difficult, since in this case the transition probabilities are continuous. This follows from the summability of the aging factors $g_j .$ The proof of Theorem \ref{theo:2} relies on a space-time Kalikow-type decomposition presented in Section \ref{sec:4}.  

\begin{rem}
The proofs of both Theorem \ref{theo:3} and Theorem \ref{theo:2} imply the existence of a perfect simulation algorithm of the stochastic chain $(X_t)_{t \in \Z} .$ By a perfect simulation algorithm we mean a simulation which samples in a finite space-time window precisely from the stationary law $P .$ In the setup of Theorem \ref{theo:2} the simulation can be implemented analogously to what is presented in Galves {\it et al.} (2013). The setup of Theorem \ref{theo:3} requires a conditional approach, conditionally on the realization of the spontaneous spike times. 
\end{rem}

\section{Correlations between inter-spike intervals in the critical directed Erd\"os-R\'enyi random graph}\label{sec:stat}
We consider a finite system consisting of a large number of $N$ neurons with random synaptic weights $ W_{i \to j} , i \neq j .$ The sequence $ W_{i \to j} , i \neq j ,$ is a sequence of i.i.d. Bernoulli random variables defined on some probability space $ ( \tilde \Omega , \tilde {\cal A} , \tilde P ) $ with parameter $p = p_N  ,$ 
i.e. 
$$ \tilde P ( W_{i \to j}  = 1) = 1  - \tilde P(  W_{i \to j}  = 0 ) = p_N ,$$
where 
\begin{equation}\label{eq:criticalregime}
 p_N = \lambda / N \mbox{ and } \lambda = 1 + \vartheta / N \mbox{ for some } 0 < \vartheta < \infty .
\end{equation}
If we represent this as a directed graph where the directed link $ i \to j$ is present if and only if $W_{ i \to j } = 1,$ we obtain what is called a ``critical directed Erd\"os-R\'enyi random graph''. For a general reference on random graphs we refer the reader to the classical book by Bollob\'as (2001)\nocite{bollobas}. 

We put $ W_{j \to j } \equiv 0 $ for all $ j .$ Notice that the synaptic weights $W_{i \to j } $ and $W_{ j \to i } $ are distinct and independent random variables. Conditionally on the choice of the connectivities $ W = ( W_{ i \to j }, i \neq j ) ,$ the dynamics of the chain are then given by
$$ P^W ( X_t (i) = 1 | {\cal F}_{t-1} ) =   \phi_i ( \sum_{ j } W_{j \to i } \sum_{ s= L_t^i }^{t- 1} g_j ( t-s) X_s (j)  ).$$
Here we suppose that $\phi_i$ is a function only of the accumulated and weighted number of spikes coming from interacting neurons, but does not depend directly on the time elapsed since the last spike. 

$P^W $ denotes the conditional law of the process, conditioned on the choice of $W .$ We write $\P $ for the annealed law where we also average with respect to the random weights, i.e. $ \P = \tilde E \left[ P^W ( \cdot ) \right] ,$ where $\tilde E $ denotes the expectation with respect to $\tilde P.$ 

Fix a neuron $i$ and consider its associated sequence of successive spike times 
\begin{equation}\label{eq:spikes}
 \ldots < S^i_{- n } < \ldots < S^i _0 \le 0 < S^i_1 < S^i_2 < \ldots < S^i_n < \ldots ,
\end{equation}
where 
$$ S_1^i = \inf \{ t \geq 1 : X_t(i) = 1 \} , \ldots , S_n^i = \inf \{ t > S_{n-1}^i : X_t (i ) = 1 \} , n \geq 2 , $$
and
$$ S_0^i = \sup \{ t \le 0 : X_t (i) = 1 \} , \ldots , S^i_{ -n} = \sup \{ t < S^i_{ - n + 1 } : X_t (i) = 1 \} , n \geq 1 .$$

Let us fix $W.$ We are interested in the covariance between successive inter-spike intervals
$$ Cov^W (  S^i_{k+1}- S^i_k , S^i_k - S^i_{k-1} ) = E^W [( S^i_{k+1}- S^i_{k} ) ( S^i_k - S^i_{k-1})] - E^W ( S^i_{k+1} - S^i_k ) E^W ( S^i_k - S^i_{k-1} )   ,$$
for any $ k \neq 0,   1 .$
Since the process is stationary, the above covariance does not depend on the particular choice of $k .$ The next theorem shows that neighboring inter-spike intervals are asymptotically uncorrelated as the number of neurons $N$ tends to infinity.  

\begin{theo}\label{theo:4}
Assume that (\ref{eq:Lip}), (\ref{eq:delta}) and (\ref{eq:G}) are satisfied. Then there exists a measurable subset $ A \in \tilde {\cal A} ,$ such that on $A,$ 
$$ |Cov^W (  S^i_3- S^i_2 , S^i_2 - S^i_1 )| \le \frac{3}{\delta^2} N (1 - \delta)^{\sqrt{N} }  , $$
where $\delta $ is the lower bound appearing in Condition (\ref{eq:delta}). 
Moreover,  
$$ \P ( A^c  ) \le e^{2\vartheta} N^{ - 1/2}  .$$

\end{theo}

For large $N,$ if the graph of synaptic weights belongs to the ``good'' set $A,$ the above result is compatible with the discussion in Gerstner and Kistler (2002). Gerstner and Kistler (2002) deduce from this that spiking can be described by a renewal process. However, for small $N$ or on $A^c ,$ the correlation might as well be quite big, as reported by Nawrot et al. (2007) who show that neighboring inter-spike intervals are negatively correlated. Therefore, both features are captured by our model, depending on the scale we are working in. 
 
The proof of the above theorem is given in Section \ref{sec:proofstat} below.

\section{Conditional Kalikow-type decomposition and proof of Theorem \ref{theo:3}}\label{sec:3}
In order to prove existence and uniqueness of a chain consistent with (\ref{eq:dyn}) and (\ref{eq:transition2}), we introduce a Kalikow-type decomposition of the infinite order transition probabilities. This type of decomposition was considered in the papers by Ferrari, Maass, Mart\'{\i}nez and Ney (2000)\nocite{ferrari-maass}, Comets, Fern\'andez and Ferrari (2002)\nocite{Comets2002} and Galves {\it et al.} (2013)\nocite{gglo2013}. All these papers deal with the case in which the transition probabilities are continuous. This is not the case here.  We are dealing with a more general case in which the transition probabilities might as well be discontinuous, see the discussion in Remark \ref{rem:1}. This makes our approach new and interesting by itself. The new ingredient is a construction of a decomposition depending on a random environment. This random environment is given by the realization of the spontaneous spikes.

More precisely, Condition (\ref{eq:delta}) allows to introduce a sequence of i.i.d. Bernoulli random variables $ (\xi_t (i ) , i \in I, t \in \Z )$ of parameter $\delta ,$ such that positions and times $(i, t) $ with $ \xi_t (i ) = 1 $ are spike times for any realization of the chain. We call these times ``spontaneous spike times''. We work conditionally on the choice of $\xi_t (i ) ,t \in \Z , i \in I  .$ In particular we will restrict everything to the state space 
$$ {\cal S}^\xi = \{ x \in \{ 0, 1 \}^{ I \times \Z } : x_t (i) \geq \xi_t(i), \forall i\in I, t \in \Z \} $$
which is the space of all configurations compatible with $\xi ,$ i.e. all neural systems $x$ such that every spike time of $\xi$ is also a spike time of $x.$ We write 
\begin{equation}\label{eq:rti}
R_t^i  = \sup \{ s<  t : \xi_s (i ) = 1 \} 
\end{equation}
for the last spontaneous spike time of neuron $i$ before time $t .$ Moreover, for $x \in {\cal S}^\xi, $ we put $L_t^i = L_t^i (x) = \sup \{ s < t : x_s (i ) = 1 \} .$

Consider a couple  $(i, t ) $ with $ \xi_t (i) = 0 .$ In order to prove Theorem \ref{theo:3} we need to introduce the following quantities which depend on the realization of $\xi ,$  namely
\begin{equation}\label{eq:314}
r_{ (i,t)}^{ [-1 ]} (1 )  = \inf_{ x \in \SS } \phi_i \left( \sum_{j } W_{j \to i} \sum_{ s = L_t^i (x) }^{t-1} g_j (t-s) \, x_s (j) , t - L_t^i (x)\right)  , \end{equation}
which is the minimal probability that neuron $i$ spikes at time $t,$ uniformly with respect to all configurations, and
\begin{equation}\label{eq:315}
 r_{ (i,t)}^{ [-1 ]} (0 )  = \inf_{ x \in \SS } \left[ 1 - \phi_i \left( \sum_{j } W_{j \to i} \sum_{ s = L_t^i (x) }^{t-1} g_j (t- s) \,  x_s (j), t- L_t^i (x)\right) \right],
\end{equation}
which is the minimal probability that neuron $i$ does not spike at time $t .$ 

Notice that for all $x \in \SS, $ $L_t^i (x) \geq R_t^i .$ Hence in the above formulas, only a finite time window of the configuration $x$ is used, and uniformly in $x \in \SS , $ this time window is contained in $(x_s (j), R_t^i \le s \le {t- 1} , j \in I )  .$ In particular the above quantities are well-defined. 

Now fix $x \in \SS .$ For any $ k \geq 0 , $ we write for short $ x_{L_t^i  }^{t-1} (V_i(k))  $ for the space-time configuration
$$  x_{L_t^i  }^{t-1} (V_i(k)) = ( x_s (j ) : L_t^i \le s \le t-1, j \in V_i (k )) $$ 
and put
\begin{equation}\label{eq:316}
 r_{ (i,t)}^{ [k ]} (1 | x_{L_t^i  }^{t-1} (V_i(k))) =
\inf_{z \in \SS   : z (V_i(k) ) = x (V_i (k)) } 
\phi_i  \left( \sum_{j } W_{j \to i} \sum_{ s = L_t^i (x) }^{t-1} g_j ( t- s ) \, z_s (j) , t- L_t^i (x) \right) ,   \end{equation} 
\begin{eqnarray}\label{eq:317}
&&r_{(i,t)}^{ [k ]} (0| x_{L_t^i }^{t-1} (V_i(k))) = \nonumber \\
&& \inf_{z \in \SS   : z (V_i(k)) = x (V_i(k)) } \left[ 1 - 
\phi_i \left( \sum_{j } W_{j \to i} \sum_{ s = L_t^i (x) }^{t-1} g_j (t-s) \, z_s (j) , t- L_t^i (x) \right)  \right]  .
\end{eqnarray}
In what follows and whenever there is no danger of ambiguity, we will write for short $ x( V_i (k)) $ and $ r_{(i,t)}^{[k]} ( a | x (V_i(k)))$ or $ r_{(i,t)}^{[k]} ( a |x) $ instead of $x_{L_i^t }^{t-1} (V_i(k))$ and $ r_{(i,t)}^{[k]} ( a | x_{L_t^i }^{t-1}  (V_i(k))).$ 

We put
\begin{equation}
\alpha_{(i,t)} ( -1) = \lambda_{(i,t)} ( -1) = \sum_{a = 0}^1 r_{(i,t)}^{[-1]} ( a ) ,
\end{equation}
\begin{equation}\label{eq:alpha}
\alpha_{(i,t)} ( k ) = \inf_{ x \in \SS}  \left( \sum_{a = 0}^1 r_{(i,t)}^{[k]} ( a| x ( V_i (k ) ) ) \right) , \; k \geq 0 ,
\end{equation}
and
\begin{equation}\label{eq:lambdak}
 \lambda_{(i,t)} (k ) = \alpha_{(i,t)} ( k )  - \alpha_{(i,t)} ( k - 1 ), k \geq 0 .
\end{equation}
Note that $ \lambda_{(i,t)} ( k ) \in [ 0, 1 ] $ and that $ \sum_{ k \geq - 1 }  \lambda_{(i,t)} ( k ) = 1 $ almost surely with respect to the realization of $ (\xi_t (i ) , i \in I, t \in \Z ).$

\begin{rem}
The $ \lambda_{(i,t)} (k ) , k \geq 0 , $ are random variables that depend only on the realization of $ (\xi_t (i ) , i \in I, t \in \Z ).$ 
More precisely, for any $i, t$ and $ k ,$  $ \lambda_{(i,t)} (k) $ is measurable with respect to the sigma-algebra $ \sigma ( \xi_s (j) , R_t^i \le s < t , j \in I ).$
We write $ \lambda_{(i,t)}^\xi ( k) , k \geq 0, $ in order to emphasize the dependance on the external field $\xi .$
\end{rem}

We introduce the short hand notation
\begin{equation}\label{eq:pit}
 p_{(i,t)} ( 1 | x) = \phi_i  \left( \sum_j W_{j\to  i } \sum_{s = L_t^i (x)  }^{t-1} g_j(t-s) \, x_s ( j) , t - L_t^i (x) \right)
\end{equation}
and 
$$ p_{(i,t)} ( 0 | x ) = 1 - p_{(i,t)} (1 | x) .$$
The proof of Theorem \ref{theo:3} is based on the following proposition. 

\begin{prop}\label{prop:1}
Given $\xi, $ for all $(i,t)$ with $ \xi_t (i ) = 0 ,$ there exists a family of conditional probabilities $  (p_{(i,t)}^{[k], \xi } (a|x))_{ k \geq 0 } $ satisfying the following properties.
\begin{enumerate}
\item
For all  $a,$ $ k \geq 0,$  $\SS \ni x \mapsto p_{(i,t)}^{[k], \xi } ( a | x ) $ depends only on the variables $ (x_s ( j ): L_t^i \le s \le t-1 , j \in V_i (k) ) .$ 
\item
For all $x \in \SS , $ $ k \geq 0,$ $ p_{(i,t)}^{[k] , \xi } ( 1 | x ) \in [0, 1 ], $ $ p_{(i,t)}^{[k], \xi } ( 0| x )  +  p_{(i,t)}^{[k], \xi} ( 1 | x ) = 1 .$
\item
For all $a, x , k \geq 0 ,$ $p_{(i,t)}^{[k], \xi } ( a | x ) $ is a $\sigma ( \xi_s (j) , R_t^i \le s \le t-1 , j \in I)-$measurable random variable. 
\item
For all $x \in \SS ,$ we have the following convex decomposition.
\begin{equation}\label{eq:dec2}
p_{(i,t)} ( a | x )  = \lambda_{(i,t)} ( -1 ) p_{(i,t) }^{[-1] , \xi } ( a) + 
\sum_{ k \geq 0 } \lambda_{(i,t)} ( k ) p_{(i,t) }^{[k] , \xi } ( a| x ( V_i (k) )) ,
\end{equation}
where
$$ p_{(i,t) }^{[-1], \xi } ( a) = \frac{ r_{(i,t)}^{[-1]} ( a)}{ \lambda_{(i,t)} ( -1 ) } .$$
\end{enumerate}
\end{prop}

From now on, we shall omit the subscript $\xi $ whenever there is no danger of ambiguity and write $ p_{(i,t)}^{[ k ]} $ instead of $ p_{(i,t) }^{[k] , \xi } .$

\begin{rem}
The decomposition (\ref{eq:dec2}) of the transition probability $p_{(i,t)} ( \cdot | x ) $ can be interpreted as follows. In a first step, we choose a random spatial interaction range $k \in \{ - 1 , 0 , 1 , \ldots \} $ according to the probability distribution $\{ \lambda_{(i,t) } (k ) , k \geq - 1 \} .$ Once the range of the spatial interaction is fixed, we then perform a transition according to $p_{(i,t) }^{[k]} $ which depends only on the finite space-time configuration $  x_{L_t^i  }^{t-1} (V_i(k)) .$ A comprehensive introduction to this technique can be found in the lecture notes of Fern\'andez et al. (2001)\nocite{FFG_EBP}.
\end{rem}

\begin{ex}
Suppose that all interactions are excitatory, i.e. $W_{j \to i } \geq 0 $ for all $ i \neq j .$ Suppose further that $ g_i (s) = 1 $ for all $ s \in \N , i \in I, $ and that $\Phi_i  (s,n) = \Phi_i ( s) $ does not depend on $n$ and is non-decreasing in $s.$ Then
$$ r_{(i,t)}^{[-1]} ( 1 ) = \Phi_i \left( \sum_{j } W_{j \to  i } \xi_{t- 1 } ( j)  \right) \; \mbox{ and } \; r_{(i,t)}^{[-1]} ( 0 ) = 1 - \Phi_
i \left( \sum_j W_{j \to  i } ( t - R_t^i ) \right) .$$ 
Hence,
$$ \lambda_{(i,t)} (-1)  = 1 + \Phi_i \left( \sum_{j } W_{j \to i } \xi_{t- 1 } ( j) \right) - \Phi_i \left( \sum_j W_{j \to i } ( t- R_t^i ) \right) .$$
Similarly,
$$ r_{(i,t)}^{ [k]} ( 1 | x ) = \Phi_i \left( \sum_{ j \notin V_i (k) } W_{j \to i } \sum_{m= L_t^i (x) }^{t- 1 } \xi_s (j) + \sum_{ j \in V_i (k)} W_{j \to i } \sum_{s= L_t^i (x)}^{t- 1 } x_s ( j) \right)  $$
and
$$ r_{(i,t)}^{ [k]} ( 0 | x ) = 1 - \Phi_i  \left( \sum_{ j \notin V_i (k) } W_{j \to i } ( t - L_t^i (x))  + \sum_{ j \in V_i (k)} W_{j \to i } \sum_{s= L_t^i (x)  }^{t- 1 } x_s ( j) \right) . $$
\end{ex}

\begin{proof}{\bf of Proposition \ref{prop:1}}
We have for any $N \geq 1, $ $ a \in \{  0, 1\}  $ and $x \in \SS ,$  
$$ p_{(i,t)} ( a| x ) = r_{(i,t)}^{[-1]} ( a) + \left( \sum_{ k= 0 }^N\Delta_{(i,t)}^{[k]} ( a | x (V_i(k))) \right) + \left( p_{(i,t)} ( a | x ) - r_{(i,t)}^{[N]} ( a | x ) \right),$$
where 
$$ \Delta_{(i,t)}^{[k]} ( a | x (V_i(k))) =  r_{(i,t)}^{ [k]} (  a | x(V_i(k))) - r_{(i,t)}^{ [k-1]} (  a | x (V_i(k-1))) .$$ 
Now, due to condition (\ref{eq:Lip}),
$$  |p_{(i,t)} ( a | x ) - r_{(i,t)}^{[N]} ( a | x )  | \le \gamma   \; \left( \sum_{ s = R_t^i }^{t- 1 } \sup_j g_j (t-s) \right) \sum_{ j \notin V_i (N) } | W_{ j \to i } |  \to 0   $$
as $ N \to \infty $ due to (\ref{eq:summable}). In the above upper bound we used that 
$$ \sum_{s = L_t^i }^{t- 1 } g_j( t-s) | z_s (j) - x_s (j) | \le \sum_{ s = R_t^i }^{t- 1 } \sup_j g_j(t-s)  < \infty $$
almost surely, which is a consequence of (\ref{eq:G}). Therefore we obtain the following decomposition.
\begin{equation}\label{eq:dec1}
p_{(i,t)} ( a | x ) = r_{(i,t)}^{[-1]} (a ) + \sum_{ k = 0 }^\infty \Delta^{[k]}_{(i,t)} ( a | x (V_i (k ))) , \, a \in  \{0, 1 \} , \mbox{ for all } x \in \SS .
\end{equation}
Now let
$$ p_{(i,t)}^{[-1]} (a) = \frac{r_{(i,t)}^{[-1]} (a)}{ \lambda_{(i,t)}  (-1)} .$$
Moreover, for $ k \geq 0,$ put 
\begin{equation}\label{eq:lambdaalmost}
 \tilde{\lambda}_{(i,t)} (k, x(V_i(k))) =  \sum_a  \Delta_{(i,t)}^{[k]} (a|  x (V_i (k )))  ,
 \end{equation}
and for any $(i,t) , k$ such that $ \tilde{\lambda}_{(i, t)}  (k,  x (V_i (k ))) > 0,$ we define
$$ \tilde{p}_{(i,t)}^{[k]} (a |  x (V_i (k )) )= \frac{\Delta_{(i,t)}^{[k]} ( a | x (V_i(k)))}{ \tilde{\lambda}_{(i,t)} (k, x (V_i (k ))) }.$$ For $(i,t) , k$ such that
$ \tilde{\lambda}_{(i, t)}  (k,  x (V_i (k )))= 0,$ define $ \tilde{p}_{(i,t)}^{[k]} (a |  x(V_i (k )) )$ in an arbitrary fixed way. Hence  
\begin{equation}\label{eq:almost}
 p_{(i,t)} (a |x ) =  \lambda_{(i,t)}  (- 1) p_{(i,t)} ^{[-1]} (a) + \sum_{k= 0 }^\infty \tilde{\lambda}_{(i,t)} (k, x (V_i(k)))
\tilde{p}_{(i,t)}^{[k]} (a |  x (V_i (k )) )  .
\end{equation}
In (\ref{eq:almost}) the factors $\tilde{\lambda}_{(i,t)} (k, x(V_i(k))),
k \geq 0, $ still depend on $x^{t-1}_{L_t^i} (V_i(k)) .$ To obtain the desired
decomposition, we must rewrite it as follows.

For any $(i,t),$ take the sequences $\alpha_{(i,t)}
(k), \lambda_{(i,t)}  (k), k \geq -1,$ as defined in (\ref{eq:alpha}) and
(\ref{eq:lambdak}), respectively. Define the new quantities
$$\alpha_{(i,t)} (k, x (V_i(k))) = \sum_{l \le k} \tilde{\lambda}_{(i,t)} (l, x (V_i(l)))$$
and notice that 
$$ \alpha_{(i,t)} (k, x (V_i(k)))  = \sum_a  r_{(i,t)}^{[k]} (a, x ( V_i ( k ))) $$
is the total mass associated to $ r_{(i,t)}^{[k]} (\cdot , x( V_i ( k ))) .$ From now on and for the rest of the proof we shall write for short
$ \alpha_{(i,t) } (k, x ) $ instead of writing $ \alpha_{(i,t)} (k, x (V_i(k))) .$ 

Reading (\ref{eq:almost}) again, this means that for any $k \geq 0 , $ we have to use the transition probability $ \tilde p_{(i,t)}^{[k]} $ on the interval $ ] \alpha_{(i,t)} ( k- 1 , x ) , \alpha_{(i,t)} (k , x)] .$ 

By definition of $ \alpha_{(i,t)} (k) $ 
in (\ref{eq:alpha}), $ \alpha_{(i,t)} (k) $ is the smallest total mass associated to $ r_{(i,t)}^{[k]},$ 
uniformly with respect to all possible neighborhoods $ x ( V_i(k)).$ Hence, in order to get
the decomposition (\ref{eq:dec2}) with weights $ \lambda_{(i,t)} (k) $ not depending on the
configuration, we have to define a partition of the interval $ [ 0, \alpha_{(i,t)}  (k, x)]$
according to the values of $ \alpha_{(i,t)} (k) $ and we have to define probabilities $p_{(i,t)}^{[k]} $ working on the intervals $ ] \alpha_{(i,t)} ( k- 1) , \alpha_{ (i,t) } ( k ) ] .$  This can be done as follows.

Fix $ k \geq 0 $ and suppose that for some $l' \le l \le k-1, $
\begin{multline*}
\alpha_{(i,t)} (l'- 1, x ) < \alpha_{(i,t)} (k-1) \le \alpha_{(i,t)} (l', x )  < \ldots <\\
< \alpha_{(i,t)} (l, x )  < \alpha_{(i,t)} (k) \le \alpha_{(i,t)} (l +1, x ) ,
\end{multline*}
and therefore
\begin{multline*} 
] \alpha_{(i,t)} ( k- 1 ) , \alpha_{(i,t)} ( k ) ] = ] \alpha_{(i,t)} (k-1) , \alpha_{(i,t)} (l', x ) ] \cup \\
\left( \bigcup_{ m = l' + 1 }^l ] \alpha_{(i,t)} ( m-1 , x) , \alpha_{(i,t)} (m, x ) ] \right) \cup \; ] \alpha_{(i,t)} ( l, x) , \alpha_{(i,t)} ( k ) ] .
\end{multline*}
Hence the probability $p_{(i,t)}^{[k]} $ that has to be defined on the interval $ ] \alpha_{(i,t)} ( k- 1) , \alpha_{ (i,t) } ( k ) ] $ has to be decomposed according to the above decomposition into sub-intervals. On the first interval $ ] \alpha_{(i,t)} (k-1) , \alpha_{(i,t)} (l', x ) ] , $ we have to use the original probability $ \tilde p^{[l']}_{(i,t)} , $ on each of the intervals $] \alpha_{(i,t)} ( m-1 , x) , \alpha_{(i,t)} (m, x ) ]  ,$ we have to use $ \tilde p_{(i,t)}^{ [m]} , $ and finally on $  ] \alpha_{(i,t)} ( l, x) , \alpha_{(i,t)} ( k ) ] ,$ we use $ \tilde p_{(i,t)}^{ [ l+1 ] } .$ 

This yields, for  any $k \geq 0$, the following definition of the conditional finite range
probability densities. 
\begin{multline}
\label{eq:pik}
 p_{(i,t)}^{[k]} ( a| x (V_i({k})))  \\
=
 \sum_{-1 = l' \le l }^{k-1}  1_{\{
\alpha_{(i,t)} (l' - 1 ,x) < \alpha_{(i,t)} ({k-1}) \le \alpha_{(i,t)} ({l'},
x)\}} 
1_{\{
\alpha_{(i,t)} (l, x) < \alpha_{(i,t)} ({k}) \le \alpha_{(i,t)} ({l+1},
x)\}}  \\
\left[ \frac{\alpha_{(i,t)} (l', x ) - \alpha_{(i,t)} (k-1) }{ 
\lambda_{(i,t)} ({k})} \; \tilde{p}_{(i,t)}^{[l']} (a | x (V_i(l')))\right.   \\
  + \sum_{m = l'+1}^{l} \frac{\tilde{\lambda}_{(i,t)}  (m , x (V_i (m))}{  \lambda_{(i,t)} (k) }
\;  \tilde{p}_{(i,t)} ^{[m]} (a | x (V_i(m))) 
  \\
\left. + \; \frac{\alpha_{(i,t)} ({k}) - \alpha_{(i,t)}  (l, x )}{
\lambda_{(i,t)} ({k})} \; \tilde{p}_{(i,t)} ^{[l+1]} (a| x (V_i({l+1}))) \right] .
\end{multline} 
Note that by construction, the above defined probability $p_{(i,t)}^{[k]} $ depends only on the configuration $ x( V_i ( l+1 ) ), $ hence at most on $ x (V_i ( k ) ) , $ since $ l \le k - 1.$ 
Multiplying the above formula with $ \lambda_{(i,t)} ( k) $ and summing over all $k$ shows that the decomposition (\ref{eq:almost}) implies the desired decomposition (\ref{eq:dec2}). This finishes our proof.   
\end{proof}

Thanks to condition (\ref{eq:Lip}), the following estimate holds. It will be useful in the sequel.

\begin{lem}
We have
for all $ k \geq 1, $ 
\begin{eqnarray}\label{eq:upperboundlambda2}
\lambda_{(i,t)}^\xi ( k) &\le&  \gamma \sum_{ j \notin V_i (k-1 ) }  | W_{j \to i } | \sum_{ s = R_t^i }^{t-1} g_j (t-s) (  1 - \xi_s (j) )  \nonumber \\
&\le &  \gamma \left( \sum_{ s = R_t^i }^{t-1} \sup_j g_j(t-s)  \right)  \sum_{ j \notin V_i (k-1 ) }  | W_{j \to i } | 
.
\end{eqnarray}
\end{lem}

We use this estimate in the

\begin{proof}{\bf of Theorem \ref{theo:3}}
We work conditionally on the realization of the process $ \xi_t (i ) , t \in \Z , i \in I .$ Take a couple $ ( i, t) , i \in I , t \in \Z ,$ such that $ \xi_t (i) = 0 .$ We have to show that it is possible to decide in a uniquely determined manner if $t$ will be a spike time or not. In order to achieve this decision, we have to calculate $ p_{(i,t)} ( \cdot | x ),$ where $x$ is the unknown history of the process. 

We will construct a sequence of sets $ (C_n^{(i,t)})_n ,$ $ C_n^{(i,t)} \subset I \times ]- \infty , t- 1] ,$ which contain the sets of sites and anterior spike times that have an influence on the appearance of a spike at time $t$ for neuron $i.$ The choice of these sets is based on the decomposition (\ref{eq:dec2}).

First, for any couple $ (j,s) $ with $ \xi_s (j) = 0 $, we choose, independently from anything else, an interaction neighborhood ${\cal V}_{(j,s)} = V_j (k) , k \geq -1 , $ with probability $ \lambda_{(j,s)}^\xi ( k) .$ Here, we put $V_j (-1) = \emptyset .$ A choice $k = -1 $ for a couple $ (j,s) $ implies that we can immediately decide to accept a spike at time $s$ for neuron $j$ with probability 
$$ \frac{r_{(j,s)}^{[-1]}(1) - \delta}{1 - \delta }$$
and to reject it with probability
$$\frac{r_{(j,s)}^{[-1]}(0) }{1 - \delta }.$$

We suppose that the choice of all ${\cal V}_{(j,s)}  $ is fixed. Then by (\ref{eq:dec2}), the decision concerning the couple $ (i,t)$ depends on the configuration of the past $ x_{L_t^i }^{t-1} ( {\cal V}_{(i,t)} ).$ Since we do not know $L_t^i , $ we use the a priori estimate on $ L_t^i $ which is given by $ R_t^i .$ Thus we consider the worst case in which we have to evaluate the past up to time $  R_t^i . $ 

In order to decide about the length $ t - L_t^i ,$ we have to assign values $0 $ or $1$ to all couples $ ( i, s ), R_t^i < s \le t .$ All these couples are influenced by their associated interaction region $ {\cal V}_{ (i,  s)} .$ So we consider
\begin{equation}\label{eq:ancestor}
C_1^{(i,t)} = \bigcup_{ s= R_t^i +1}^{t} \bigcup_{j \in {\cal V}_{(i,s)} \setminus \{ i\} } \; \bigcup_{ u =  R_t^i  }^{ s-1} \{ (j, u ) , \xi_j (u) = 0 \} .
\end{equation}

It is clear that if we know the values of all couples belonging to $ C_1^{(i,t)} $ then we are able to assign a value to any $(i,  s), R_t^i <  s \le t .$ Therefore, we call $C_1^{(i,t)}$ the set of ancestors of generation $1$ of $(i,t).$
Continuing this procedure, any couple $ ( j,s) \in C_1^{ (i,t)} $ itself has to be replaced by the set of its ancestors
$ C_1^{(j,s)} .$ We put 
\begin{equation}\label{eq:evol}
 C_2^{(i,t)} = \left( \bigcup_{ ( j, s ) \in C_1^{ (i,t)} }  C_1^{(j,s )}  \right)  \setminus C_1 ^{ (i,t)} ,
\end{equation}
where $C_1^{(j,s )}$ is defined as in (\ref{eq:ancestor}), and then recursively,
\begin{equation}
C_n^{(i,t)}= \left( \bigcup_{ ( j, s ) \in C_{n-1} ^{ (i,t)} }  C_1^{(j,s )}  \right)  \setminus \left( C_1^{ (i,t)} \cup \ldots \cup C_{n-1}^{(i,t)} \right) .
\end{equation}
We will show below that almost surely there exists a first time $n< \infty $ such that $C_n^{(i,t)} = \emptyset .$ Then necessarily $C_{n-1}^{(i,t)} $ consists only of couples $(j,s)$ which chose an interaction neighborhood of range $ - 1$ or which interact only with couples representing spontaneous spiking. Thus we can decide to accept or to reject a spike for any of them independently of anything else. Once the values associated to all elements of $ C_{n-1}^{(i,t)} $ are known, we can realize all decisions needed in order to attribute values to the elements of $C_{n-2}^{(i,t)} .$ In this way, we will be able to assign values in a recursive way to all ancestor sets up to the first set, $C_{1 }^{(i,t)} ,  $
which allows us finally to assign a value for neuron $i$ at time $t.$ We call this procedure the forward coloring procedure and we call $Y_t (i) $ the value of neuron $t$ at time $t$ obtained at the end of this procedure. 

$1.$ We first show that the probability constructed above is the law of $X_t(i)$ under the unique invariant measure $P.$ Put $C_\infty^{(i,t)} = \bigcup_{ n \geq 1 } C_n^{(i,t)} $ and let 
\begin{equation}\label{eq:tstop}
T^{(i,t)}_{STOP} = \inf \{ s : C_\infty^{(i,t)}  \subset  I \times [ t - s , t- 1 ] \} .
\end{equation}

Fix some initial space-time configuration $\eta \in \{ 0, 1 \}^{ I \times  - \N  }$ such that $ \eta_{0} ( i ) = 1 $ for all $i \in I.$ Let $X_t^\eta $ be the chain that evolves according to (\ref{eq:dyn}) and (\ref{eq:transition2}), conditionally on $ X_{- \infty}^0 = \eta .$ Recall that $Y_t (i)$ is the value obtained at the end of the above described forward coloring procedure. Then for any $ f: \{ 0, 1 \}  \to \R_+ ,$ 
\begin{eqnarray}
 E ( f ( X_t^\eta (i) )  &= &E ( f ( X_t^\eta (i) ), T^{(i,t)}_{STOP} < t  )
+ E ( f ( X_t^\eta (i) ) , T^{(i,t)}_{STOP} \geq t ) \nonumber \\
& =  & E ( f ( Y_t (i) ), T^{(i,t)}_{STOP} < t  ) \nonumber \\
&& \quad \quad  \quad \quad  \quad \quad  \quad \quad  + E (f ( X_t^\eta (i) ) , T^{(i,t)}_{STOP} \geq t ) .
\end{eqnarray}
But 
$$ E ( f ( X_t^\eta (i) ) , T^{(i,t)}_{STOP} \geq t ) \le \| f \|_{\infty} P ( T^{(i,t)}_{STOP} \geq t) \to 0 \mbox{ as } t \to \infty .$$ 
Hence we obtain that 
$$ \lim_{t \to \infty} E ( f ( X_t^\eta (i) ) ) = E ( f ( Y_t (i) ) ) ,$$
since $ 1_{\{   T^{(i,t)}_{STOP} < t \}  } \to 1 $ almost surely. 

This implies that $P $ is the unique invariant measure of the process. 

$2.$ We now show that the above procedure stops after a finite time almost surely. 
For that sake we put
\begin{equation}
N_{STOP}^{(i,t)} = \min \{ n : C_n^{(i,t)} = \emptyset \} .
\end{equation}
Our goal is to show that $ N_{STOP}^{(i,t)} < \infty $ almost surely. Notice that this implies that $ T_{STOP}^{(i,t)} < \infty $ as well. In order to do so, let
$$  | C_n^{(i,t)} | , n \geq 1 , $$
be the cardinal of the set of ancestors after $n$ steps of the above procedure. Then
\begin{equation}\label{eq:rough}
 P (  N_{STOP}^{(i,t)}  > n ) \le E (  | C_n^{(i,t)} | ) .
\end{equation}
As a consequence, it is sufficient to show that $E( | C_n^{(i,t)} |) \to 0 $ as $n \to \infty .$ For that sake we compare $( | C_n^{(i,t)} |)_n$ to a branching process
in random environment, where the environment is given by the i.i.d. field $  \xi .$ 

Recall the definition of $G(n) = \sum_{ m=1}^{n} \sup_i g_i (m)$ and the upper bound (\ref{eq:upperboundlambda2}). Thus for any $ k \geq 1,$ 
\[
 \lambda_{(i,t)} (k ) \le   \gamma \left( \sum_{ s = R_t^i }^{t-1} \sup_j g_j (t- s ) \right)    \sum_{ j \notin V_i (k-1) } | W_{ j \to i }|
=  \gamma  G ( t - R_t^i ) \sum_{ j \notin V_i (k-1) } | W_{ j \to i }| . 
\]
Let 
\begin{equation}\label{eq:lambdabar} 
\bar \lambda_{(i,t)} (k) =  \gamma  G ( t - R_t^i ) \sum_{ j \notin V_i (k-1) } | W_{ j \to i }| .
\end{equation}
Notice that $ \bar \lambda_{(i,t)} (k) $ depends only on the realization of the i.i.d. field $\xi$ through the value of $t- R_t^i .$

Since $\lambda_{(i,t)} (k ) \le \bar \lambda_{(i,t)} (k)  $ for all $ k \geq 1 ,$ it is possible, by standard branching process coupling arguments, to construct a sequence  $ (  \bar C_n^{(i,t)})_n $ coupled with the original sequence $ (  C_n^{(i,t)})_n $ satisfying the following properties.

\begin{enumerate}
\item
$\bar C_n^{(i,t)} $ is defined through (\ref{eq:ancestor}) and (\ref{eq:evol}) by using the family $ (\bar \lambda_{(i,t)} (k ) ) $ instead of $ ( \lambda_{(i,t)} (k ) ) .$
\item
For all $n,$  $  |C_n^{(i,t)}| \le |\bar C_n^{(i,t)}|.$
\end{enumerate}
Hence it is sufficient to show that $ E ( |\bar C_n^{(i,t)} | ) $ tends to $0$ as $n \to \infty .$ In order to do so, we will control the reproduction mean depending on the environment $\xi .$ Here, reproduction mean stands for the mean number of sites belonging to $ \bar C_1^{(j,s)},$ for any fixed couple $(j,s),$ where the expectation is taken with respect to the choices of the interaction neighborhoods $ \bar {\cal V}_{(i,s)}, $ conditionally on the realization of $\xi.$ More precisely, given $\bar C_n^{(i,t)} = c,$ the reproduction mean of any couple $ (j, s)$ belonging to $\bar C_n^{(i,t)} $ is given by
\begin{multline*}
 E^\xi \left( |\bar C_1^{(j,s)} |\,  \Big| \, (j,s) \in \bar C_n^{(i,t)} = c \right)  \\
 \le \sum_{ \tilde s = R_s^j + 1 }^s  \sum_{k \geq 1 } \bar \lambda_{(j,\tilde s)} (k ) \left( \sum_{ l \in V_j (k ) , l \neq j }  \sum_{ u = R_s^j }^{ \tilde s- 1} [1- \xi_{ u} ( l) ] 1 _{\{ (l, u ) \notin c \}} \right) ,
\end{multline*}
where $E^\xi $ denotes expectation conditionally on $\xi.$

In what follows we upper bound the above expression. We first use that, by definition (\ref{eq:lambdabar}), for $ \tilde s \le s,$ $ \bar \lambda_{(j,\tilde s)} (k )  \le \bar \lambda_{(j,  s)} (k ) .$ Therefore,
\begin{multline*}
  E^\xi \left( |\bar C_1^{(j,s)} |\,  \Big| \, (j,s) \in \bar C_n^{(i,t)} = c \right)
\\
\le  (s- R_s^j ) \sum_{k \geq 1 } \bar \lambda_{(j,  s)} (k ) \left( \sum_{ l \in V_j (k ) , l \neq j } \sum_{ u = R_s^j  }^{  s- 1} [1- \xi_{ u} ( l) ] 1 _{\{ (l, u ) \notin c \}} \right) .
\end{multline*}

Recalling the explicit form of $\bar \lambda_{(j,s)} (k)  $ in (\ref{eq:lambdabar}) and using moreover the upper bound 
$$ \sum_{ u = R_s^j  }^{  s- 1} [1- \xi_{ u} ( l) ] 1 _{\{ (l, u ) \notin c \}}  \le s - R_j^s\, , $$ 
on the event $ R_j^s < s- 1 ,$ we obtain 
\begin{multline}\label{eq:mjc}
 E^\xi \left( |\bar C_1^{(j,s)} | \, \Big| \, (j,s) \in \bar C_n^{(i,t)} = c \right)  \\
 \le  2 \gamma   
\left(   1_{\{  R_s^j < s-1 \} }  ( s - R_s^j )^2 
  G (s-R_s^j)   \sum_{ k \geq 1 } | V_j (k) | \left( \sum_{ l \notin V_j (k-1) } | W_{ l \to  j }|  \right) \right.\\
\left. + 1_{\{  R_s^j =  s-1 \} }    G  (1)   \sum_{ k \geq 1 } \left( \sum_{ l \notin V_j (k-1) } | W_{ l \to  j }|  \right)   \sum_{ l \in V_j (k ) , l \neq j }  ( 1 - \xi_{s-1}  (l ) ) 1 _{\{ (l, s-1 ) \notin c \}}  \right)
 \\
= : \bar m_{(j,s)}^{ c }.
\end{multline}
Observe that $\bar m_{(j,s)}^{ c }$ depends on $\xi,$ but to avoid too cumbersome notation, we omit the superscript $\xi .$ 

Taking conditional expectation, conditionally with respect to $\xi, $ we get
\begin{equation}\label{eq:quenched}
 E^\xi \left( |\bar C_n^{(i,t)} | \,  \Big| \, \bar C_{n-1}^{(i,t)}  \right) \le \sum_{ (j, s ) \in  \bar C_{n-1}^{(i,t)}} \bar m^{  \bar C_{n-1}^{(i,t)} }_{ (j,s)} .
\end{equation}

Therefore,  taking expectation with respect to $\xi ,$
\begin{equation}\label{eq:tobeit}
E( |\bar C_n^{(i,t)} | ) \le \sum_{(j,s)} E \left( 1_{(j,s) \in  \bar C_{n-1}^{(i,t)} }  \bar m^{  \bar C_{n-1}^{(i,t)} }_{ (j,s)} \right)  .
\end{equation}
Now note that the event
$$ \{ (j,s) \in \bar C_{n-1}^{(i,t)}  \} = \{ \exists (k, \tilde s ) \in \bar C_{n-2}^{(i,t)}  : \tilde s > s , k \neq j,  j \in \bar{\cal V}_{(k, \tilde s )} \} $$
depends only on $ R_{\tilde s }^k .$ Hence, recalling (\ref{eq:mjc}), the above event is independent of $\bar m_{(j,s)}^{ \bar C_{n-1}^{(i,t)}} $ which depends only on $ s - R_s^j .$ As a consequence,
$$  E \left( 1_{(j,s) \in  \bar C_{n-1}^{(i,t)} }  \bar m^{  \bar C_{n-1}^{(i,t)} }_{ (j,s)} \right) = P \left( (j,s) \in  \bar C_{n-1}^{(i,t)} \right) E (\bar m^{  \bar C_{n-1}^{(i,t)} }_{ (j,s)} ) .$$
But
\begin{multline}\label{eq:abm}
  E( \bar m^{  \bar C_{n-1}^{(i,t)} }_{ (j,s)}   )    \le
  2 \gamma   
\left[ \sum_{n=2}^\infty \delta (1 - \delta)^{ n-1} n^2 G( n) \sum_{ k \geq 1 }  | V_j (k) | \left( \sum_{ l \notin V_j (k-1) } | W_{ l \to  j }|  \right)  \right. \\
 \left. +\delta  G  (1)   \sum_{ k \geq 1 } \left( \sum_{ l \notin V_j (k-1) } | W_{ l \to j }|  \right) 
  | V_j (k ) | (1 - \delta)    \right] .
\end{multline}
We put
\begin{equation}\label{eq:cgamma}
 C_\gamma := 2 \gamma \sup_j \sum_{ k \geq 1 } | V_j (k) | \left( \sum_{l \notin V_j (k-1) }  | W_{l \to j}| \right) 
\end{equation}
and define 
$$E(G, \delta ) = G(1) + \sum_{ n = 2 }^{ \infty} ( 1 - \delta)^{n-2} n^2 G(n) .$$ 
Then
\begin{equation}\label{eq:ub3}
  E[ \bar m^{  \bar C_{n-1}^{(i,t)} }_{ (j,s)}   ]    \le      C_\gamma \;  (1- \delta) E( G, \delta) =: e(\delta) .
\end{equation}
Note that $e(\delta) $ is decreasing as a function of $ \delta $ and tends to $0$ for $ \delta \to 1.$ In particular, if we put 
\begin{equation}\label{eq:deltastar}
\delta^* = \inf \{ \delta \in ] 0, 1 [ : e ( \delta ) \le 1 \}, 
\end{equation}
then $e(\delta ) < 1 $ for all $ \delta \geq \delta_* . $
Now we may iterate (\ref{eq:tobeit}) and (\ref{eq:ub3}) and obtain
\begin{equation}\label{eq:geo}
 E ( |\bar C_n^{(i,t)} |  ) \le \sum_{ (j,s) } E ( 1_{\{ (j,s) \in \bar C_{n-1}^{(i,t)} \} } )e(\delta) = E ( |\bar C_{n-1}^{(i,t)} | ) e(\delta ) \le e(\delta)^n  .
\end{equation}
Since $e(\delta )^n \to 0 $ as $n \to \infty $ for $ \delta \geq \delta_* , $ this implies our result.
\end{proof}

\begin{rem}
Note that the value of $\delta^* $ given in (\ref{eq:deltastar}) can be explicitly calculated depending on the specific structure of the aging functions $g_j ( s) .$ For instance, if $ g_j ( s) = 1 $ for all $j, s, $ then $G(n) = n $ for all $n,$ and the value of $\delta^* $ follows from standard evaluations of the third moment of a geometrical distribution.
\end{rem}

\begin{rem}
The above proof uses a non-trivial extension of the so-called Clan of ancestors method employed by Fern\`andez et al. (2001)\nocite{MR1849182} and Fern\`andez et al. (2002)\nocite{ferr:fern:garc:2002}. In these papers the authors study the Clan of ancestors of a given vertex (or object) in a random system. They prove that this set is almost surely finite and deduce a perfect simulation algorithm.
\end{rem}

\begin{proof}{\bf of Corollary \ref{cor:mixing} part 1.}
We prove the first part of Corollary \ref{cor:mixing}.
We keep the notation of the proof of Theorem \ref{theo:3}. Put $C_\infty^{(i, t)} = \bigcup_n C_n^{(i,t)} .$ Then the random variable $ T_{STOP}^{(i,t)} $ defined in (\ref{eq:tstop}) can be trivially upper bounded by 
\begin{equation}\label{eq:total}
 T_{STOP}^{(i,t)} \le   |C_\infty^{(i,t)} | ,
\end{equation}
which is the total number of elements appearing in the ancestor process. This is a very rough upper bound on the number of steps that we have to look back into the past in order to choose a value for $ X_i (t).$ 

By construction, the value $X_i (t) $ depends on all choices of interaction regions $ {\cal V}_{ (j, s)} $ such that $ (j, s) \in C_\infty^{(i, t)} $ and on the values of the i.i.d. Bernoulli field $ \xi_u $ for $ t - T_{STOP}^{(i,t)}  \le u \le t .$ Moreover, for any $ (j,s) \in C_\infty^{(i, t)} , $ a random decision has to be made whether to associate the value $ + 1 $ or $0$ to this couple. These decisions can be realized based on a sequence of i.i.d. uniform random variables $ (U_t (i), i \in I, t \in \Z ),$ which are uniform on $[0, 1 ] .$ As a consequence, $X_t (i) $ is measurable with respect to the sigma-algebra
$$ \sigma \{ {\cal V}_{ (j, s)} , U_s (j ) : (j, s) \in C_\infty^{(i, t)}  , \xi_u : t - T_{STOP}^{(i,t)}  \le u \le t \}  .$$
Writing 
\begin{equation}
 R = \inf_{ u \in [s, t]} ( u - T_{STOP}^{(i,u)} ) , 
\end{equation}
we deduce that 
\begin{equation}\label{eq:ms}
\mbox{
$X_s^t (i) $ is $ \sigma \{  {\cal V}_{ (j, u)} ,U_u (j),  \xi_u (j ) : j \in I, R  \le u \le t \} -$measurable.}
\end{equation} 
Let $ A \in \F_0 .$ We write $E^\xi$ for conditional expectation, conditionally with respect to $\xi .$ Then we have 
\begin{eqnarray*}
E [ f ( X_s^t (i)) 1_A ] &=& E [ f ( X_s^t (i)) 1_A 1_{ \{ R > 0 \}} ] + E [ f ( X_s^t (i)) 1_A 1_{ \{ R \le 0 \}} ]\\
&=& E [ E^\xi [  f ( X_s^t (i)) 1_A 1_{ \{ R > 0 \}} ] ] +
E [ E^\xi [ f ( X_s^t (i)) 1_A 1_{ \{ R \le 0 \}} ]] .
\end{eqnarray*}
Clearly, $ f ( X_s^t (i)) 1_{\{ R > 0 \}} $  and $ A$ are independent, which follows from (\ref{eq:ms}). Hence 
$$ E^\xi [ f ( X_s^t (i)) 1_A ] = E^\xi [ f( X_s^t (i)) 1_{ \{ R > 0 \}} ] P^\xi (A) + E^\xi [ f(X_s^t (i)) 1_A 1_{\{ R \le 0 \} } ] .$$
We take expectation with respect to $\xi $ and use the fact that $ E^\xi [ f( X_s^t (i)) 1_{ \{ R > 0 \}} ] $ is measurable with respect to $ \sigma \{ \xi_u : 1 \le u \le t \} $ and 
$ P^\xi (A) $ is measurable with respect to $\sigma \{ \xi_u : u \le 0 \} .$ Hence by independence
$$ E [  f ( X_s^t (i)) 1_A ] = E [ f( X_s^t (i)) 1_{ \{ R > 0 \}} ] P( A) + E [ f(X_s^t (i)) 1_A 1_{\{ R \le 0 \} } ],$$
and therefore
$$
E [  f ( X_s^t (i)) 1_A ]  - E[ f (X_s^t (i))] P(A) =   E [ f(X_s^t (i)) 1_A 1_{\{ R \le 0 \} } ] - E [ f( X_s^t (i)) 1_{ \{ R \le 0 \}} ] P( A) .
$$
As a consequence,
\begin{eqnarray*}
 \left| E [  f ( X_s^t (i)) 1_A ]  - E[ f (X_s^t (i))] P(A) \right| 
& \le& \| f \|_\infty P( A) \left[ P ( R\le 0 | A ) + P (R\le 0 ) \right] \\
&\le & \| f \|_\infty P( A) \left[ P ( R\le 1 | A ) + P (R\le 0 ) \right]  \\
&\le & 2 \| f \|_\infty P( A) P( R \le 1 ) .
\end{eqnarray*}
Here we have used that $ \{ R \le 1 \}$ is independent of $A.$ Indeed, the event $\{ R \le 1 \} $ does only depend on choices of ${\cal V}_{ (j,s)} $ and $\xi_s (j) $ having time component $ s \geq 1 .$ Now we conclude as follows. By definition of $R,$ 
$$ P ( R \le 1 ) \le \sum_{ u \in [s, t]} P ( T_{STOP}^{ (i,u)} \geq s - 1 ) . 
$$
Using (\ref{eq:total}), we obtain
$$ P ( T_{STOP}^{ (i,u)} \geq s - 1 ) \le P (  |C_\infty^{(i,u)} | \geq s - 1 )\le E \left(  |C_\infty^{(i,u)} |\right) \frac{1}{ s-1} \le \frac{1}{ 1 - e(\delta)} \frac{1}{ s-1}  , $$
since
$$ E \left(  |C_\infty^{(i,u)} |\right) \le \sum_{ n= 1 }^\infty E ( |C^{(i,u)}_n| ) \le \frac{1}{ 1 - e(\delta)}  ,$$
where we have used (\ref{eq:geo}). This implies the result.
\end{proof}

\section{A space-time Kalikow-type decomposition and proof of Theorem \ref{theo:2}}\label{sec:4}
The additional condition (\ref{eq:veryveryfast}) allows to introduce a space-time Kalikow-type decomposition without the assumption of existence of spontaneous spikes. We will decompose with respect to increasing space-time neighborhoods $ V_i (k ) \times [- k - 1 , - 1 ] , k \geq - 1 ,$ where $ V_i ( - 1 ) = \emptyset.$  Write ${\cal S} = \{0, 1 \}^{I \times \Z} $ for the state space of the process and introduce, analogously to (\ref{eq:314})--(\ref{eq:317}), 
\begin{equation}
r_i^{[-1]} ( 1) = \inf_{ x \in {\cal S}} \Phi_i ( \sum_j W_{ j \to i } \sum_{ s = L_0^i (x) }^{-1} g_j (-s) x_s (j)  ) ,
\end{equation}
\begin{equation}
r_i^{[-1]} ( 0) = \inf_{ x \in {\cal S}} \left( 1 - \Phi_i ( \sum_j W_{ j \to i } \sum_{ s = L_0^i (x) }^{-1} g_j (-s) x_s (j)  ) \right) ,
\end{equation}
\begin{equation}
r_i^{[0]} ( 1 |x) = \inf_{ z\in {\cal S} : z_i ( - 1) = x_i (-1) } \Phi_i ( \sum_j W_{ j \to i } \sum_{ s = L_0^i (z) }^{-1} g_j (-s) z_s (j)  ) ,
\end{equation}
\begin{equation}
r_i^{[0]} ( 1| x) = \inf_{ z \in {\cal S} : z_i ( - 1) = x_i (-1)  } \left( 1 - \Phi_i ( \sum_j W_{ j \to i } \sum_{ s = L_0^i (z) }^{-1} g_j (-s) z_s (j)  ) \right) ,
\end{equation}
and then for any $ k \geq 1 , $ 
\begin{equation}
r_i^{[k]} ( 1 |x) = \inf_{ z\in {\cal S} : z ( V_i (k) \times [- k  - 1 , - 1 ] ) =  x( V_i (k) \times [- k - 1 , - 1 ] )   } \Phi_i ( \sum_j W_{ j \to i } \sum_{ s = L_0^i (z) }^{-1} g_j (-s) z_s (j) ) ,
\end{equation}
\begin{equation}
r_i^{[k]} ( 0 |x) = \inf_{ z\in {\cal S} : z ( V_i (k) \times [- k  - 1 , - 1 ] ) =  x( V_i (k) \times [- k - 1 , - 1 ] )   } \left( 1 - \Phi_i ( \sum_j W_{ j \to i } \sum_{ s = L_0^i (z) }^{-1} g_j (-s) z_s (j) ) \right)  .
\end{equation}
Putting 
$$ \alpha_i (- 1) = \lambda_i ( - 1 ) = r_i^{[-1]} ( 1)  + r_i^{[-1]} ( 0 ) , $$
\begin{equation}\label{eq:lambdakbis}
 \alpha_i ( k ) = \inf_x \left(  r_i^{[k]} ( 1 |x) + r_i^{[k]} ( 0  |x) \right) , \lambda_i (k ) = \alpha_i ( k ) - \alpha_i ( k- 1 ) , k \geq 0 , \end{equation}
we obtain the following space-time Kalikow-type decomposition for 
$$ p_{(i,t)} ( 1 | x ) = \Phi_i \left( \sum_{ j } W_{j \to i } \sum_{ s= L_t^i (x)}^{t- 1 } g_j (t-s) x_s (j)  \right) .$$
\begin{prop}
Under the conditions (\ref{eq:summable}), (\ref{eq:Lip}) and (\ref{eq:veryveryfast}), for any $ i \in I,$ the above defined quantities $\lambda_i (k ) $ take values in $[0, 1 ]$ and 
$$ \sum_{ k \geq -1} \lambda_i ( k) = 1 .$$
Moreover, there exists a family of conditional probabilities $ (p_{(i,t)}^{[k]} ( a |x ) )_{ k \geq 0 } $ satisfying the following properties.
\begin{enumerate}
\item
For all $a,$ $ {\cal S} \ni x \mapsto p_{(i,t)}^{[0] } ( a | x ) $ depends only on the variable $ x_{t- 1 } (i) .$ 
\item
For all  $a,$ $ k \geq 1 ,$  $ {\cal S} \ni x \mapsto p_{(i,t)}^{[k] } ( a | x ) $ depends only on the variables $ (x_s ( j ):t - k - 1  \le s \le t-1 , j \in V_i (k) ) .$ 
\item
For all $x \in  {\cal S} , $ $ k \geq 0,$ $ p_{(i,t)}^{[k] } ( 1 | x ) \in [0, 1 ], $ $ p_{(i,t)}^{[k] } ( 0| x )  +  p_{(i,t)}^{[k]} ( 1 | x ) = 1 .$
\item
For all $x \in {\cal S},$ we have the following convex decomposition
\begin{equation}\label{eq:dec3}
p_{(i,t)} ( a | x )  = \lambda_{i} ( -1 ) p_{(i,t) }^{[-1]  } ( a) + 
\sum_{ k \geq 0 } \lambda_{i } ( k ) p_{(i,t) }^{[k]  } ( a| x) ,
\end{equation}
where
$$ p_{(i,t) }^{[-1] } ( a) = \frac{ r_{i}^{[-1]} ( a)}{ \lambda_{i} ( -1 ) } .$$
\end{enumerate}
\end{prop}

The proof of this proposition follows the lines of the proof of Proposition \ref{prop:1}. Moreover, the following estimates hold.
\begin{equation}\label{eq:upperboundlambda1new}
\lambda_{i} ( 0) \le  \gamma   \sum_{ j  }  | W_{j \to i } |  \sum_{ n = 1}^{\infty} g_j (n)    ,
\end{equation}
and for all $ k \geq 1, $ 
\begin{equation}\label{eq:upperboundlambda2new}
\lambda_{i } ( k) \le   \gamma \left( \sum_{ j \notin V_i (k-1 ) }  | W_{j \to i } | \sum_{ n = 1  }^{\infty} g_j (n) + \sum_{ j \in V_i (k-1 ) }  | W_{j \to i } | \sum_{ n = k  }^{\infty} g_j (n) \right).
\end{equation}

\begin{proof}{\bf of Theorem \ref{theo:2}}
We construct again a sequence of sets $ (C_n^{(i,t)})_n  \subset I \times ]- \infty , t- 1] $ which contain the sets of sites and anterior spike times that have an influence on the appearance of a spike at time $t$ for neuron $i.$ The choice of these sets is based on the decomposition (\ref{eq:dec3}).

First, we choose for any couple $ (j,s) $, independently from anything else, a space-time interaction neighborhood ${\cal O}_{(j,s)}  \subset I \times [ - \infty, s- 1 ] $ 
$${\cal O}_{(j,s)} = \left\{ 
\begin{array}{ll}
V_j (k) \times [ s- k- 1 , s- 1 ]& \mbox{ with probability $ \lambda_{j}( k)  , k \geq 1 ,$} \\
\{ j \} \times \{ s- 1 \} & \mbox{ with probability $ \lambda_{j}( 0) $} \\ 
\emptyset & \mbox{ with probability $ \lambda_{j}( -1 ) .$}
\end{array}
\right. $$
Then we put
\begin{equation}
C_1^{(i,t)} = {\cal O}_{(i,t)} , \; C_n^{(i,t)} = \left( \bigcup_{ (j, s) \in C_{n-1}^{(i,t)}} {\cal O}_{(j,s)}  \right) \setminus \left( C_1^{(i,t)} \cup \ldots \cup C_{n-1}^{(i,t)}\right) .
\end{equation}
The process $ | C_n^{(i,t)} | $ can be compared to a classical multi-type branching process with reproduction mean 
$$ m_i = \lambda_i (0 ) + \sum_{ k \geq 1 } (k+1)  | V_i (k ) | \lambda_i ( k ) .  $$
Now, condition (\ref{eq:veryveryfast}) together with the estimates (\ref{eq:upperboundlambda1new}) and (\ref{eq:upperboundlambda2new}) shows that $ m := \sup_i m_i < 1 .$ As a consequence, 
$$   P( N_{STOP}^{(i,t) } > n ) \le E (| C_n^{(i,t)} |) \le m^n \to 0 \mbox{ as } n \to \infty , $$
and this finishes the proof.
\end{proof}

\begin{proof}{\bf of Corollary \ref{cor:mixing}}
The proof of the first part of Corollary \ref{cor:mixing} is analogous to the proof under the conditions of Theorem \ref{theo:3}. We give the proof of the second part of the Corollary. We keep the notation of the proof of Theorem \ref{theo:2}. We define the projection on the time coordinate
\begin{equation}
T( {\cal O}_{(j,s)}) := s- k- 1 \mbox{ if  } {\cal O}_{(j,s)} = V_j (k) \times [s- k - 1  , s- 1 ] , k \geq 0 ,  T( {\cal O}_{(j,s)}) := s \mbox{ else.}
\end{equation}
Define recursively
$$ T_1^{(i,t)} = T( {\cal O}_{(i,t)}) , T_n^{(i,t)} = \min \{ T( {\cal O}_{(j,s)}) : (j,s) \in C_{n-1}^{(i,t)} \} , n \geq 1 ,$$
and let finally
$$ T^{(i,t) }_{STOP} := T_{N^{(i,t)}_{STOP}- 1 }^{(i,t) } .$$
Then as in the proof of the first part of the corollary,
$$ \left| E [  f ( X_s^t (i)) 1_A ]  - E[ f (X_s^t (i))] P(A) \right| 
\le  2 \| f \|_\infty P( A) P( R \le 1 ) ,$$
where 
$$ R = \inf_{ u \in [s, t ]} T_{STOP}^{(i,u)} .$$
We have
$$ P ( R \le 1 ) \le \sum_{ u \in [s, t ] } P( T^{(i,u)}_{STOP} \le 1 ) ,$$
where
$$ P (  T^{(i,u)}_{STOP} \le 1 )  = P ( u-  T^{(i,u)}_{STOP} \geq  u- 1 ) .$$
In what follows, $C$ will denote a constant that might change from line to line, but that does not depend on $\beta .$ 

We wish to compare $u- T^{(u,t) }_{STOP}$ to the total offspring of a classical Galton-Watson branching process. In order to do so, 
notice first 
that using (\ref{eq:upperboundlambda1new}) and (\ref{eq:upperboundlambda2new}), we get for any $ k \geq 1 , $ 
$$ \lambda_i ( k )\le C \frac{e^{ - \beta (k-1) }}{1 - e^{ - \beta } } =: \bar \lambda (k)  ,$$
where we have used (\ref{eq:exponentialg}). Moreover, 
$$ \lambda_i ( 0 ) \le  C \frac{e^{ - \beta  }}{1 - e^{ - \beta } } =: \bar \lambda (0 )  .$$
It is immediate to see that for $\beta $ sufficiently large, $ \sum_{ k \geq 0} \bar \lambda (k) < 1 ,$ since 
$$ \sum_{ k \geq 0} \bar \lambda (k) = \frac{C}{ 1 - e^{ - \beta } } \left[ e^{ - \beta} + \frac{e^{- \beta}}{1 - e^{ - \beta } } \right] \to 0 $$
as $\beta \to \infty .$ 
Therefore we can couple $u- T^{(i,u) }_{STOP}  $ with the total offspring $ {\cal T}$ of a classical Galton-Watson process, where each particle has $k + 1 $ offspring, $ k \geq 0 ,$ with probability $ \bar \lambda  (k ) , $ and offspring $ 0$ with probability $ 1 - \sum_{ k \geq 0 } \bar \lambda ( k) .$ This coupling can be done such that $u- T^{(i,u) }_{STOP}  \le {\cal T} .$ Thus it is sufficient to evaluate the law of the total offspring ${\cal T}.$ For that sake let $W_n$ be a random walk starting from $0$ at time $0, $ with step size distribution $ \eta , $ where 
$$ \eta  = \left\{ 
\begin{array}{ll}
- 1 & \mbox{ with probability }  1 - \sum_{ k \geq 0 } \bar \lambda ( k) \\
k & \mbox{ with probability } \bar \lambda (k ) , k \geq 0 .
\end{array}
\right. $$
Then $ {\cal T} \stackrel{\cal L}{=} T_{-1} = \inf \{ n : W_n = - 1 \} .$ Hence for any $\lambda \in ] 0 , \beta [  , $ since $ u \geq s,$
\begin{multline*}
 P ( u-  T^{(i,u)}_{STOP} \geq  u- 1 ) \le P ( T_{ - 1 } \geq u - 1 )  \le P ( T_{ - 1 } \geq s - 1 ) = P ( T_{ - 1 } > s - 2 )\\
 \le P (  e^{\lambda W_{ s- 2 }} \geq 1 ) \le \varphi ( \lambda )^{s-2} , 
\end{multline*}
where 
\begin{equation}\label{eq:onlyg}
\varphi ( \lambda ) = E ( e^{\lambda \eta })  = e^{ - \lambda } [ 1 - \sum_{ k \geq 0 } \bar \lambda ( k)] + \frac{C e^{ - \beta}}{1 - e^{ - \beta} } + 
\frac{C}{1-e^{ - \beta} } \frac{e^{ \lambda -  \beta }}{1 - e^{ \lambda - \beta}} .
\end{equation}
Now, fix $ \lambda = 1 , $ then there exists $\beta_* $ such that for all $\beta \geq \beta_*, $ $\varphi ( 1) < 1 .$ Putting $ \varrho = \varphi ( 1 ) $ and $ C = \varrho^{ - 2 } $ yields the desired result.
\end{proof}

\section{Proof of Theorem \ref{theo:4}}\label{sec:proofstat}
In order to prove Theorem \ref{theo:4}, we introduce the sequence of sets
$$ {\cal V}^1_{i \to \cdot }  = \{ j : W_{ i \to j } = 1 \} , \ldots , {\cal V}^n_{i \to \cdot } = \{ j : \exists k \in {\cal V}^{n-1}_{i \to \cdot } : W_{ k \to j } = 1 \} , n \geq 2 .$$
Note that $j \in {\cal V}_{i \to \cdot}$ if and only if neuron $i$ has a direct influence on the spiking behavior of neuron $j.$
We put 
$$ \tau^i = \inf \{ n : i \in {\cal V}^n_{i \to \cdot } \} .$$
This is the first time that an information emitted by neuron $i$ can return to neuron $i$ itself.

Recall that $\lambda =  1 + \vartheta / N $ and define $ \mu =  \frac{N-2}{N} \lambda .$ We have the following lower bound.

\begin{prop}\label{prop:tau}
For any $k$ the following inequality holds.
$$ \tilde P ( \tau^i \le  k ) \le \frac{k-1}{N} \, \exp \left({ \vartheta \frac{k}{N} }\right).$$
\end{prop}

\begin{proof}
Put 
$$ \tilde {\cal V}^n_{i \to \cdot } = \{ j \neq i : \exists k \in \tilde {\cal V}_{i \to \cdot }^{ n-1} : W_{ k \to j } = 1 \} , n \geq 2 , \tilde {\cal V}_{i \to \cdot }^1 = {\cal V}^1_{i \to \cdot } .$$
The sequence of sets $ \tilde {\cal V}^n_{i \to \cdot } , n \geq 1 , $ equals the original sequence $ {\cal V}^n_{i \to \cdot } , n \geq 1 , $ except that we excluded the choice of $i$ itself. On $ \{ \tau^i > k   \} ,$ clearly $ \bigcup_{ n \le k- 1 }  {\cal V}^n_{i \to \cdot } = \bigcup_{ n \le k- 1 } \tilde {\cal V}^n_{i \to \cdot },$ and we can write
$$ \tilde P( \tau^i >  k ) = \tilde P \left( W_{ j \to i } = 0 \; \; \forall \; j \in \bigcup_{ n \le k- 1 } \tilde  {\cal V}^n_{i \to \cdot }\right) .$$
Since in the definition of $\tilde {\cal V}^n_{i \to \cdot } ,$ no choice $ W_{ \cdot \to i } $ has been made, we can condition with respect to $\bigcup_{ n \le k- 1 } \tilde {\cal V}^n_{i \to \cdot } , $ use the fact that for any $j \in \bigcup_{ n \le k- 1 } \tilde {\cal V}^n_{i \to \cdot }, $ the random variable $W_{ j \to i } $ is independent of $\bigcup_{ n \le k- 1 } \tilde {\cal V}^n_{i \to \cdot } ,$ and obtain the following equality
$$ \tilde P( \tau^i > k ) = \tilde E \left[ (1 - p_N )^{ | \bigcup_{1 \le  n \le k- 1 } \tilde {\cal V}^n_{i \to \cdot } | } \right] .$$
We conclude as follows. We can couple the process $ | \tilde {\cal V}^n_{i \to \cdot } |, n \geq 2 , $ with a classical Galton-Watson process $Z_n , n \geq 2 ,$ starting from $Z_1 = {\cal V}^1_{i \to \cdot } ,$ such that $ | \tilde {\cal V}^n_{i \to \cdot } | \le Z_n $ for all $n \geq 2 .$ The Galton-Watson process has offspring mean $ \mu = (N- 2 ) \frac{ \lambda}{N} .$ Here, the factor $N-2$ comes from the fact that any $j$ has $N- 2 $ choices of choosing arrows $ W_{ j \to \cdot } ,$ since $ j $ itself and $i$ are excluded. 

Therefore, 
$$ \tilde P( \tau^i > k ) = \tilde  E \left[ (1 - p_N )^{ | \bigcup_{ 1 \le n \le k- 1 } \tilde {\cal V}^n_{i \to \cdot } | } \right] \geq \tilde E \left[ (1 - p_N )^{\sum_{n = 1 }^{ k - 1 } Z_n } \right] .$$
Write $ \Sigma_{ k - 1} = Z_1 + \ldots + Z_{k-1}  $ and let $ \tilde E (s^{\Sigma_{k-1}}), s \le 1  ,$ be its moment generating function. Using the convexity of the moment generating function, we have that 
$$ \tilde E (s^{\Sigma_{k-1}}) \geq 1 + \tilde E ( \Sigma_{k-1}) (s- 1 ) .$$
Using that $\tilde E ( Z_1 ) = \frac{N-1}{N} \lambda $ and that the offspring mean equals $\mu,$ the claim follows from
$$ \tilde E ( \Sigma_{k-1}) = \frac{N-1}{N} \lambda \left[1+ \mu +\ldots + \mu^{k-2} \right]\le \lambda + \ldots + \lambda^{ k - 1}  \le (k-1) \lambda^{k-1}  ,$$
since $\mu \le \lambda $ and $\lambda \geq 1.$ 
Hence, evaluating the above lower bound in $s  = 1 - p_N ,$ we obtain
$$  \tilde P( \tau^i > k ) \geq 1 - p_N (k-1) \lambda^{k-1},$$
and therefore, 
$$ \tilde P( \tau^i \le k ) \le p_N (k-1) \lambda^{ k- 1} =  \frac{k-1}{N}  \lambda^k ,$$
since $ p_N = \lambda / N .$ Using that $ \lambda = 1 + \vartheta / N , $ we obtain the assertion. 
\end{proof}

In what follows, $ a_{-k}^{-1} $ denotes the finite sequence $  ( a_{-k}, \ldots , a_{-1} ). $ In particular, the notation  $a_{-l}^{-1} 1 0^{k-1}  $ denotes the sequence given by $ ( a_{-l} , \ldots , a_{-1}, 1, 0 , \ldots , 0 ) .$ We write for short 
$$  p^{(W,i)} ( a | a_{-k}^{-1} ) =  P^W ( X_{k} (i) = a | X_0^{ k-1 }(i ) = a_{-k}^{-1} ) $$
for the transition probability of neuron $i,$ given a fixed choice of synaptic weights $W.$ However, conditionings will be read from the left to the right. In particular, we write
$$  p^{(W,i)} ( a | 0^{k-1} 1 a_{-l}^{-1} ) =  P^W ( X_{k} (i) = a | X_{k-1}( i ) = \ldots = X_1 (i ) = 0 , X_0 (i) = 1 , X_{-l}^{-1} ( i ) =  a_{-l}^{-1} ). $$
The following proposition shows that on the event $ \{ \tau^i > k + l \} ,$ the two transition probabilities $  p^{(W,i)} (1 | 0^{k-1} 1 a_{-l}^{-1} )  $ and $p^{(W,i)} ( 1 | 0^{k-1} 1)$ necessarily coincide.

\begin{prop}\label{prop:4}
For any $ k \geq 1 , l \geq 1 , $ 
$$ \{ p^{(W,i)} (1 | 0^{k-1} 1 a_{-l}^{-1} ) \neq  p^{(W,i)} ( 1 | 0^{k-1} 1) \} \subset \{ \tau^i \le k + l \} .$$
\end{prop}   

\begin{proof}
Let $W$ be fixed. From now on, since we will work for this fixed choice of $W,$ we will omit the superscript $W$ and write for short $p^i ( a | a_{-k}^{-1} ) $ instead of $ p^{(W,i)} ( a | a_{-k}^{-1} )  .$ We have
\begin{multline*}
 P^W( X_{k}(i) = 1 ,  X_1^{k-1} (i) = 0^{k-1} , X_0 (i) = 1 ,  X_{-l}^{-1} (i) = a_{-l}^{-1}  ) \\
= \sum_{ j \in {\cal V}_i } \sum_{z_{0}^{k-1} (j) \in \{ 0, 1 \}^{k } }P^W( X_{k}(i) = 1 ,  X_1^{k-1} (i) = 0^{k-1} , X_0 (i) = 1 ,  X_{-l}^{-1} (i) = a_{-l}^{-1} ,\\
 X_{0}^{k-1} (j ) = z_{0}^{k-1} (j) , \forall j \in {\cal V}_i )  \\
= \sum_{ j \in {\cal V}_i } \sum_{z_{0}^{k-1} (j) \in \{ 0, 1 \}^{k} } \phi_i \left( \sum_{ j \in {\cal V}_i} \sum_{ s = 0}^{k-1} g_j ( k - s ) z_s (j )  ) \right) \times\\
\times P^W (  X_{0}^{k-1} (j ) = z_{0}^{k-1} (j) , \forall j \in {\cal V}_i , X_{ - l }^{k-1} (i)= a_{-l}^{-1} 1 0^{k-1}) .
\end{multline*}
Thus,
\begin{multline*}
 p^i (1 | 0^{k-1} 1 a_{-l}^{-1} )  
= \sum_{ j \in {\cal V}_i } \sum_{z_{0}^{k-1} (j) \in \{ 0, 1 \}^{k} } \phi_i \left( \sum_{ j \in {\cal V}_i} \sum_{ s = 0}^{k-1} g_j ( k - s ) z_s (j )  ) \right)  \times \\
\times  P^W (  X_{0}^{k-1} (j ) = z_{0}^{k-1} (j) , \forall j \in {\cal V}_i | X_{ - l }^{k-1} = a_{-l}^{-1} 1 0^{k-1} ).
\end{multline*}
The same calculus shows that 
\begin{multline*}
 p^i ( 1 | 0^{k-1} 1) 
= \sum_{ j \in {\cal V}_i } \sum_{z_{0}^{k-1} (j) \in \{ 0, 1 \}^{k} } \phi_i \left( \sum_{ j \in {\cal V}_i} \sum_{ s = 0}^{k-1} g_j ( k - s ) z_s (j )  ) \right)  \cdot \\
\cdot  P^W (  X_{0}^{k-1} (j ) = z_{0}^{k-1} (j) , \forall j \in {\cal V}_i | X_{ 0 }^{k-1}(i) = 1 0^{k-1}  ).
\end{multline*}

This shows that in order to ensure that $ p^i (1 | 0^{k-1} 1 a_{-l}^{-1} ) = p^i ( 1 | 0^{k-1} 1) ,$ it is sufficient to have
\begin{multline}\label{eq:conditionsuff}
 P^W (  X_{0}^{k-1} (j ) = z_{0}^{k-1} (j) , \forall j \in {\cal V}_i | X_{ 0 }^{k-1}(i) =   1 0^{k-1} ) =\\
 =P^W (  X_{0}^{k-1}(j ) = z_{0}^{k-1} (j) , \forall j \in {\cal V}_i | X_{ - l }^{k-1}(i) = a_{-l}^{-1} 1 0^{k-1} ) ,
\end{multline}
for all possible choices of $ z_{0}^{k-1} (j) , j \in {\cal V}_i ,$ which is implied by $ \tau^i > k + l .$  
\end{proof}

{\bf Proof of Theorem \ref{theo:4}}
In this proof, without loss of generality and to simplify the presentation, we suppose that $\tilde \Omega $ is the canonical state space of $W.$ We will use the spontaneous spike times $ \{ n \in \Z : \xi_n (i) = 1 \} $ introduced in the proof of Theorem \ref{theo:3}, in Section \ref{sec:3} above. We recall that these are independent Bernoulli random variables with $ P( \xi_n (i ) = 1 ) = \delta $ for all $ i \in \{ 1, \ldots , N \} , $ for all $ n \in \Z.$ 
Write $ l = \sup  \{ n < S^i_2 : \xi_n (i) = 1 \} $ and 
$ r = \inf \{ n > S^i_2 : \xi_n (i ) = 1 \} .$ Put 
$$ A =  \{ \tau^i > 2k(N)    \}  , $$
where $k(N) $ is such that $ k( N ) \to \infty $ as $ N \to \infty $ and $ k(N) \le N .$  We will fix the choice of $ k( N) $ later. We have for any realization of $W \in  A,$ 
\begin{multline}\label{eq:1}
E^W [ (S^i_3 - S^i_2) (S^i_2 - S^i_1)   ]  \le E^W [ ( r- S^i_2) (S^i_2 -  l) 1_{ \{  l < S^i_2 - k(N) \} \cup \{  r > S^i_2 + k (N) \}  }   ]  
\\
+E^W [ ( S^i_3- S^i_2) (S^i_2 -  S^i_1) 1_{ \{  l \geq  S^i_2 - k(N) ; r \le  S^i_2 + k (N) \}  }   ]  .
\end{multline}
Using that conditionally on $S^i_2 , $ $r- S^i_2$ and $S^i_2 -  l$ are independent and geometrically distributed, we obtain a first upper bound
\begin{equation}\label{eq:2}  E^W [ ( r- S^i_2) (S^i_2 -  l) 1_{ \{  l < S^i_2 - k(N) \} \cup \{  r > S^i_2 + k (N) \}  }   ]   \le \frac{1}{\delta^2} ( k(N) + 2 ) ( 1 - \delta)^{k(N)} . 
\end{equation}
We now consider the second term and use that $\tau^i > 2 k(N) .$ We have
\begin{multline}\label{eq:3}
E^W [ ( S^i_3- S^i_2) (S^i_2 -  S^i_1) 1_{ \{  l \geq  S^i_2 - k(N) ;  r \le  S^i_2 + k (N) \}  }   ]\\
= \sum_t E^W [ ( S^i_3 - t ) ( t - S^i_1 ) 1_{ \{  l \geq  t - k(N) ;  r \le  t + k (N) \}  } 1_{ \{ S^i_2 = t \} }    ]
\\
= \sum_t  E^W \left[  ( t - S^i_1 ) 1_{ \{  l\geq  t - k(N) \}}1_{ \{ S^i_2 = t \} } E^W [  ( S^i_3 - t ) 1_{ \{   r \le  t + k (N) \}  } | {\cal G}_{t- k(N) }^{ t} ] \right],
\end{multline}
where 
$$ {\cal G}_{t- k(N) }^{ t} = \sigma \{ X_s (i ) : t- k(N) \le s \le t \} .$$
Now, since $S^i_3 \le r,$ 
\begin{multline*}
 E^W [  ( S^i_3 - t ) 1_{ \{   r\le  t + k (N) \}  } | {\cal G}_{t- k(N) }^{ t} ] = \\
\sum_{ n = 1 }^{k(N)} n \times P^W ( S^i_3 - t = n ; r \le  t + k (N)  | {\cal G}_{t- k(N) }^{ t}) 
\le \sum_{ n = 1 }^{k(N)} n \times P^W ( S^i_3 - t = n  | {\cal G}_{t- k(N) }^{ t} ) .
\end{multline*}
Notice that 
\begin{multline*}
 P^W ( S^i_3 - t = n  | {\cal G}_{t- k(N) }^{ t} ) \\
= p^i ( 0 | 1 X^{t-1}_{t- k(N) } )p^i ( 0 | 0 1 X^{t-1}_{t- k(N) } ) \times \ldots \times p^i ( 0 |0^{n-2} 1 X^{t-1}_{t- k(N) } ) p^i ( 1 |0^{n-1} 1 X^{t-1}_{t- k(N) } ).
\end{multline*}
Now we use Proposition \ref{prop:4}. Since we are working on $ \{ \tau^i > 2 k(N) \} ,$  we have 
$$ p^i ( 0 | 1 X^{t-1}_{t- k(N) } ) = p^i ( 0 | 1  ), \ldots , p^i ( 1 |0^{n-1} 1 X^{t-1}_{t- k(N) } ) = p^i ( 1 |0^{n-1} 1 ),$$
for all $ n \le k(N) .$ Therefore, 
\begin{multline}\label{eq:4}
 E^W [  ( S^i_3 - t ) 1_{ \{   r \le  t + k (N) \}  } | {\cal G}_{t- k(N) }^{ t} ] \\
\le \sum_{ n = 1 }^{k(N)} n  \times p^i ( 0 | 1 )p^i ( 0 | 0 1  ) \times \ldots \times p^i ( 0 |0^{n-2} 1   ) p^i ( 1 |0^{n-1} 1   ) \\
\le \sum_{ n = 1 }^{\infty } n  \times p^i ( 0 | 1 )p^i ( 0 | 0 1  ) \times \ldots \times p^i ( 0 |0^{n-2} 1   ) p^i ( 1 |0^{n-1} 1   )
\\= E^W ( S^i_3 - S^i_2 ) .
\end{multline}
We conclude that on $A,$ using successively (\ref{eq:1})--(\ref{eq:4}),
$$ E^W [ ( S^i_3 - S^i_2 ) (S^i_2 - S^i_1) ] \le \frac{1}{\delta^2} ( k(N) + 2 ) ( 1 - \delta)^{k(N)} + E^W ( S^i_3 - S^i_2) E^W ( S^i_2 - S^i_1).$$

In a second step, we are seeking for lower bounds. We start with 
\begin{multline}\label{eq:11}
E^W [ (S^i_3 - S^i_2) (S^i_2 - S^i_1)   ]  \geq E^W [ ( S^i_3- S^i_2) (S^i_2 -  S^i_1) 1_{ \{  l \geq  S^i_2 - k(N) ;  r \le  S^i_2 + k (N) \}  }   ] .  
\end{multline}
Then on $\{ S_2^i = t \} ,$
\begin{multline*}
 E^W [  ( S^i_3 - t ) 1_{ \{  r \le  t + k (N) \}  } | {\cal G}_{t- k(N) }^{ t} ] = 
\sum_{ n = 1 }^{k(N)} n \times P^W ( S^i_3 - t = n ; r \le  t + k (N)  | {\cal G}_{t- k(N) }^{ t}) \\
\geq \left( \sum_{ n = 1 }^{k(N)} n \times P^W ( S^i_3 - t = n  | {\cal G}_{t- k(N) }^{ t} ) \right) - k(N)^2 P^W ( r >  t + k (N)  | {\cal G}_{t- k(N) }^{ t}) \\
= \left( \sum_{ n = 1 }^{k(N)} n \times P^W ( S^i_3 - t = n  | {\cal G}_{t- k(N) }^{ t} ) \right)  - k(N)^2  (1 - \delta)^{ k(N) } .
\end{multline*}
Now, on $\{ S_2^i = t \} ,$
\begin{multline*}
 \sum_{ n = 1 }^{k(N)} n \times P^W ( S^i_3 - t = n  | {\cal G}_{t- k(N) }^{ t} ) 
\\ = E^W ( S^i_3 - S^i_2 ; S^i_3 - S^i_2 \le k(N) ) =  E^W (S^i_3 - S^i_2 ) -  E^W ( S^i_3 - S^i_2 ; S^i_3 - S^i_2 >  k(N) ) \\
\geq  E^W (S^i_3 - S^i_2 ) -E^W ( r - S^i_2 ; r- S^i_2 >  k(N) )\\
\geq  E^W (S^i_3 - S^i_2 ) - \frac{1}{\delta} (k(N) + 2 ) (1 - \delta )^{k(N) } .
\end{multline*}
Therefore, for any realization $W \in A,$
$$ E^W [ (S^i_3 - S^i_2 ) ( S^i_2 - S^i_1 ) ] \geq E^W ( S^i_3 - S^i_2 )E^W ( S^i_3 - S^i_2 )   - [ \frac{2}{\delta^2} (k(N) + 2 )  + k(N)^2  ] (1 - \delta)^{ k(N) } .
$$
Putting things together and supposing that $ k(N) + 2 \le k(N)^2 , $ we obtain finally
$$| E^W [ (S^i_3 - S^i_2 ) ( S^i_2 - S^i_1 ) ] -  E^W ( S^i_3 - S^i_2 )E^W ( S^i_3 - S^i_2 ) | \le \frac{3}{\delta^2} k(N)^2 (1 - \delta)^{k(N) } .$$

It remains to find an upper bound for $ \P ( A^c ) .$ Clearly, applying Proposition \ref{prop:tau}, since $k(N) \le N, $ we have
$$ \P( A^c ) \le e^{2\vartheta} \frac{ k(N) }{N}  .$$
It is enough to choose $k(N) = \sqrt{N} $ to conclude the proof. \hfill $\bullet$

\section*{Acknowledgments}

We thank D. Brillinger, B. Cessac, S. Ditlevsen, M. Jara, M. Kelbert, Y. Kohayakawa, C. Landim, R. I. Oliveira, S. Ribeiro, L. Triolo, C. Vargas and N. Vasconcelos for many  discussions on Hawkes processes, random graphs and neural nets at the beginning of this project.

This work is part of USP project ``Mathematics, computation,language and the brain", FAPESP project ``NeuroMat" (grant 2011/51350-6), USP/COFECUB project
``Stochastic systems with interactions of variable range'' and CNPq project ``Stochastic modeling of the brain activity" (grant 480108/2012-9). 
AG is partially supported by a CNPq fellowship (grant 309501/2011-3), AG and EL have been partially supported by the MathAmSud project ``Stochastic structures of large interacting systems" (grant 009/10). EL thanks Numec, USP, for hospitality and support.

\bibliography{biblio-15-01-2013}

\begin{thebibliography}{10}

\bibitem{bollobas}
B\'ela Bollob\'as.
\newblock {\em {Random graphs. 2nd ed.}}
\newblock {Cambridge Studies in Advanced Mathematics. 73. Cambridge: Cambridge
  University Press.}, 2001.

\bibitem{BremaudMassoulie96}
Pierre Br\'emaud and Laurent Massouli\'e.
\newblock {Stability of nonlinear Hawkes processes.}
\newblock {\em Ann. Probab.}, 24(3):1563--1588, 1996.

\bibitem{Brillinger1988}
D.~Brillinger.
\newblock {Maximum likelihood analysis of spike trains of interacting nerve
  cells}.
\newblock {\em Biol. Cybern.}, 59(3):189--200, 1988.

\bibitem{cessac_discrete_time}
B.~Cessac.
\newblock A discrete time neural network model with spiking neurons: I{I}:
  Dynamics with noise.
\newblock {\em Journal of Mathematical Biology}, 62:863--900, 2011.

\bibitem{Comets2002}
F.~Comets, R.~{Fern\'andez}, and P.A. Ferrari.
\newblock Processes with long memory: Regenerative construction and perfect
  simulation.
\newblock {\em Ann. of Appl. Probab.}, 12(3):921--943, 2002.

\bibitem{doefor37}
W.~Doeblin and R.~Fortet.
\newblock Sur les cha\^{\i}nes \`{a} liaisons compl\'{e}tes.
\newblock {\em Bull. Soc. Math. France}, 65:132--148, 1937.

\bibitem{FFG_EBP}
R.~Fern{\'a}ndez, P.~A. Ferrari, and A.~Galves.
\newblock Coupling, renewal and perfect simulation of chains of infinite order,
  {N}otes for a minicourse given at the {V}th {B}razilian {S}chool of
  {P}robability.
\newblock {\em Manuscript},
  http://www.univ-rouen.fr/LMRS/Persopage/Fernandez/resucoup.html, 2001.

\bibitem{MR1849182}
R.~Fern{\'a}ndez, P.~A. Ferrari, and N.~L. Garcia.
\newblock Loss network representation of {P}eierls contours.
\newblock {\em Ann. Probab.}, 29(2):902--937, 2001.

\bibitem{ferr:fern:garc:2002}
R.~Fern\'andez, P.A. Ferrari, and N.L. Garcia.
\newblock Perfect simulation for interacting point processes, loss networks and
  {Ising} models.
\newblock {\em Stoc. Proc. Appl.}, 102(1):63--88, 2002.

\bibitem{fern_maillard}
R.~Fern{\'a}ndez and G.~Maillard.
\newblock Chains with complete connections and one-dimensional {G}ibbs
  measures.
\newblock {\em Electron. J. Probab.}, 9:no. 6, 145--176 (electronic), 2004.

\bibitem{ferrari-maass}
P.~A. Ferrari, A.~Maass S., Mart\'{\i}nez, and P.~Ney.
\newblock {Ces\`aro mean distribution of group automata starting from measures
  with summable decay.}
\newblock {\em Ergodic Theory Dyn. Syst.}, 20(6):1657--1670, 2000.

\bibitem{Gallo2011}
S.~Gallo.
\newblock Chains with unbounded variable length memory: perfect simulation and
  visible regeneration scheme.
\newblock {\em Adv. Appl. Probab.}, 43:735--759, 2011.

\bibitem{gglo2013}
A.~Galves, N.~Garcia, E.~L\"ocherbach, and E.~Orlandi.
\newblock Kalikow-type decomposition for multicolor infinite range particle
  systems, 2013.

\bibitem{Gerstner}
W.~Gerstner and W.~M. Kistler.
\newblock {\em {Spiking neuron models. Single neurons, populations,
  plasticity.}}
\newblock {Cambridge: Cambridge University Press.}, 2002.

\bibitem{Goldberg}
J.~M. Goldberg, H.~O. Adrian, and F.~D. Smith.
\newblock Response of neurons of the superior olivary complex of the cat to
  acoustic stimuli of long duration.
\newblock {\em J. Neurophysiol.}, pages 706--749, 1964.

\bibitem{har55}
T.~E. Harris.
\newblock On chains of infinite order.
\newblock {\em Pacific J. Math.}, 5:707--24, 1955.

\bibitem{hawkes71}
Alan~G. Hawkes.
\newblock {Point spectra of some mutually exciting point processes.}
\newblock {\em J. R. Stat. Soc., Ser. B}, 33:438--443, 1971.

\bibitem{krumin}
M.~Krumin, I.~Reutsky, and S.~Shoham.
\newblock Correlation-based analysis and generation of multiple spike trains
  using {H}awkes models with an exogenous input.
\newblock {\em Front Comput Neurosci.}, 4, 2010.

\bibitem{MollerRasmussen}
Jesper M{\o}ller and Jakob~G. Rasmussen.
\newblock {Perfect simulation of Hawkes processes.}
\newblock {\em Adv. Appl. Probab.}, 37(3):629--646, 2005.

\bibitem{rissanen}
J.~Rissanen.
\newblock A universal data compression system.
\newblock {\em IEEE Trans. Inform. Theory}, 29(5):656--664, 1983.

\bibitem{Spitzer1970}
F.~Spitzer.
\newblock Interaction of {M}arkov {P}rocesses.
\newblock {\em Adv. Math.}, 5:246--290, 1970.

\end{thebibliography}

\end{document}